\documentclass[a4paper,12pt]{article}

\usepackage{typearea}
\typearea{15}

\usepackage{graphicx}
\usepackage{color}
\usepackage{mymacros}

\newcommand{\QH}{\mathop{QH}\nolimits}
\newcommand{\HF}{\mathop{H\!F}\nolimits}

\newcommand{\bslambda}{{\boldsymbol{\lambda}}}

\newcommand{\ad}{\operatorname{ad}}
\newcommand{\FS}{\mathrm{FS}}
\newcommand{\Fix}{\operatorname{Fix}}
\newcommand{\Fl}{\operatorname{Fl}}
\newcommand{\ii}{\sqrt{-1}}
\newcommand{\im}{\operatorname{Im}}
\newcommand{\Int}{\operatorname{Int}}
\newcommand{\Jac}{\operatorname{Jac}}
\newcommand{\PD}{\operatorname{PD}}
\newcommand{\po}{\mathfrak{PO}}
\newcommand{\re}{\operatorname{Re}}

\newcommand{\scMwhat}{\widehat{\scM}_{\mathrm{weak}}}

\newcommand{\reg}{\mathrm{reg}}

\renewcommand{\GL}{\operatorname{GL}}

\newcommand{\U}{\operatorname{U}}
\newcommand{\SU}{\operatorname{SU}}

\newcommand{\dge}{\rotatebox[origin=c]{45}{$\ge$}}
\newcommand{\uge}{\rotatebox[origin=c]{315}{$\ge$}}
\newcommand{\dgne}{\rotatebox[origin=c]{45}{$>$}}
\newcommand{\ugne}{\rotatebox[origin=c]{315}{$>$}}
\newcommand{\deq}{\rotatebox[origin=c]{45}{$=$}}
\newcommand{\ueq}{\rotatebox[origin=c]{315}{$=$}}
\newcommand{\updots}{\rotatebox[origin=c]{45}{$\cdots$}}
\newcommand{\dndots}{\rotatebox[origin=c]{315}{$\cdots$}}
\newcommand{\rueq}{\rotatebox[origin=c]{45}{$=$}}
\newcommand{\lueq}{\rotatebox[origin=c]{315}{$=$}}
\newcommand{\lmd}[1]{\hbox to1.65em{$\hfill \lambda_{#1} \hfill$}}
\newcommand{\llmd}[2]{\hbox to1.65em{$ \lambda_{#2}^{(#1)}$}}
\newcommand{\gcbox}[1]{\hbox to1.65em{$\hfill {#1} \hfill$}}

\title{Floer cohomologies of non-torus fibers\\
of the Gelfand-Cetlin system}
\author{Yuichi Nohara and Kazushi Ueda}
\date{}

\begin{document}
\maketitle

\begin{abstract}
The Gelfand-Cetlin system has non-torus Lagrangian fibers
on some of the boundary strata
of the moment polytope.
We compute Floer cohomologies of such
non-torus Lagrangian fibers in the cases of 
the 3-dimensional full flag manifold and 
the Grassmannian of 2-planes in a 4-space.
\end{abstract}

%\tableofcontents

\section{Introduction}
\label{sc:intro}

Let $P$ be a parabolic subgroup of $\GL(n,\bC)$ and
$F := \GL(n,\bC)/P$ be the associated flag manifold.
The Gelfand-Cetlin system, introduced by
Guillemin and Sternberg \cite{Guillemin-Sternberg_GCS},
is a completely integrable system
\[
  \Phi : F \longrightarrow \bR^{(\dim_{\bR} F)/2},
\]
i.e., a set of functionally independent and Poisson commuting functions.
The image $\Delta = \Phi (F)$ is a convex polytope
called the \emph{Gelfand-Cetlin polytope}, and
$\Phi$ gives a Lagrangian torus fibration
structure over the interior $\Int \Delta$ of $\Delta$.
Unlike the case of toric manifolds
where the fibers over the relative interior
of a $d$-dimensional face of the moment polytope
are $d$-dimensional isotropic tori,
the Gelfand-Cetlin system has non-torus Lagrangian fibers
over the relative interiors of some of the faces of $\Delta$.

Let $(X, \omega)$ be a compact toric manifold of $\dim_{\bC} X = N$,
and $\Phi: X \to \bR^N$ be the toric moment map
with the moment polytope
$\Delta = \Phi(X)$.
For an interior point $\bsu \in \Int \Delta$, let $L(\bsu)$ denote 
the Lagrangian torus fiber $\Phi^{-1}(\bsu)$.
Lagrangian intersection Floer theory endows
the cohomology group $H^*(L(\bsu); \Lambda_0)$ over the Novikov ring
\[
 \Lambda_0
%  = \Lambda_0^{\bC}
  := \left\{ \left.
      \sum_{i=1}^\infty a_i T^{\lambda_i} \, \right| \,
       a_i \in \bC, \ 
%       \lambda_i \in \bR, \ 
       \lambda_i \ge 0, \ 
       \lim_{i \to \infty} \lambda_i = \infty
     \right\}
\]
with a structure $\{ \frakm_k \}_{k \ge 0}$
of a unital filtered $A_{\infty}$-algebra
\cite{Fukaya-Oh-Ohta-Ono}.
Let $\Lambda$ and $\Lambda_+$ be the quotient field
and the maximal ideal of the local ring $\Lambda_0$ respectively.
An odd-degree element $b \in H^{\mathrm{odd}}(L(\bsu); \Lambda_0)$
is said to be a \emph{bounding cochain}
if it satisfies the 
\emph{Maurer-Cartan equation}
\begin{equation}
 \sum_{k=0}^\infty \frakm_k(b^{\otimes k}) = 0.
\end{equation}
A solution $b \in H^{\mathrm{odd}}(L(\bsu); \Lambda_0)$ to the 
\emph{weak Maurer-Cartan equation}
\begin{align}
 \sum_{k=0}^\infty \frakm_k(b^{\otimes k}) \equiv 0 \mod \Lambda_0 \, \bfe_0
\end{align}
is called a \emph{weak bounding cochain},
where $\bfe_0$ is the unit in $H^*(L(\bsu); \Lambda_0)$.
The set of weak bounding cochains will be denoted by
$\scMwhat (L(\bsu))$.
The {\em potential function} is a map
$
 \po \colon \scMwhat (L(\bsu)) \to \Lambda_0
$
defined by
\begin{align}
 \sum_{k=0}^\infty \frakm_k(b,\ldots,b) = \po(b) \bfe_0.
\end{align}
A weak bounding cochain gives a deformed filtered $A_\infty$-algebra
whose $A_\infty$-operations are given by
\begin{align}
 \frakm^b_k(x_1,\ldots,x_k)
  = \sum_{m_0=0}^\infty \cdots \sum_{m_k=0}^\infty
   \frakm_{m_0+\cdots+m_k+k}
    (b^{\otimes m_0} \otimes x_1 \otimes b^{\otimes m_1} \otimes
    \cdots \otimes x_k \otimes b^{\otimes m_k}).
\end{align}
The weak Maurer-Cartan equation
%and the $A_\infty$-relation
implies that $\frakm_1^b$ squares to zero,
and the \emph{deformed Floer cohomology} is defined by
\begin{align}
 \HF((L(\bsu),b), (L(\bsu),b); \Lambda_0)
  = \left. \Ker(\frakm_1^b) \right/ \Image(\frakm_1^b).
\end{align}
More generally, one can deform the Floer differential $\frakm_1$ 
by 
\begin{equation}
  \delta_{b_0, b_1} (x) = \sum_{k_0, k_1 \ge 0}
    \frakm_{k_0 + k_1 + 1}( \underbrace{b_0, \dots, b_0}_{k_0}, x, 
    \underbrace{b_1, \dots, b_1}_{k_1})
\end{equation}
for a pair $(b_0, b_1)$ of weak bounding cochains
with $\po(b_0) = \po(b_1)$.
The Floer cohomology of the pair
$((L(\bsu), b_0), (L(\bsu), b_1))$
is defined by
\begin{align}
 \HF((L(\bsu),b_0), (L(\bsu),b_1); \Lambda_0)
  = \left. \Ker(\delta_{b_0, b_1}) \right/ \Image(\delta_{b_0, b_1}).
\end{align}
If the toric manifold $X$ is Fano,
then the following hold
\cite{FOOO_toric_I}:
\begin{itemize}
 \item
$H^1(L(\bsu); \Lambda_0)$ is contained in $\scMwhat (L(\bsu))$.
 \item
The potential function $\po$ on
\begin{equation} \label{eq:SYZ_mirror}
 \bigcup_{\bsu \in \Int \Delta} H^1(L(\bsu); \Lambda_0/ 2\pi \ii \bZ)
  \cong \Int \Delta \times (\Lambda_0/ 2\pi \ii \bZ)^N
\end{equation}
can be considered as a Laurent polynomial,
which can be identified with the superpotential
of the Landau-Ginzburg mirror of $X$.
 \item
Each critical point of $\po$ corresponds to a pair $(\bsu, b)$
%of  $\bsu \in \Int \Delta$ and $b \in H^1(L(\bsu); \Lambda_0/2\pi \ii \bZ)$
such that the deformed Floer cohomology
$
 \HF((L(\bsu),b),(L(\bsu),b);\Lambda)
$
over the Novikov field $\Lambda$
is non-trivial.
 \item
If the deformed Floer cohomology group
over the Novikov field
is non-trivial,
then it is isomorphic to the classical cohomology group;
\begin{align}
 \HF((L(\bsu),b),(L(\bsu),b);\Lambda)
  \cong H^*(T^N ;\Lambda).
\end{align}
 \item
The quantum cohomology ring
$\QH(X; \Lambda)$
is isomorphic to the Jacobi ring $\Jac (\po)$
of the potential function.
\end{itemize}

In particular, the number of pairs $(L(\bsu), b)$ with nontrivial 
Floer cohomology coincides with
$
 \rank \QH(X;\Lambda)
  = \rank H^*(X;\Lambda).
$

Nishinou and the authors
\cite{Nishinou-Nohara-Ueda_TDGCSPF}
introduced the notion of a toric degeneration
of an integrable system,
and used it to compute the potential function of Lagrangian torus fibers
of the Gelfand-Cetlin system.
The resulting potential function can be considered as a Laurent polynomial
just as in the toric Fano case,
which can be identified with the superpotential
of the Landau-Ginzburg mirror 
of the flag manifold given in \cite{Givental_SPIQTLFMMC,
Batyrev-Ciocan-Fontanine-Kim-van_Straten_MSTD}.
%and Batyrev, Ciocan-Fontanine, Kim, van Straten 
%\cite{Batyrev-Ciocan-Fontanine-Kim-van_Straten_MSTD}.  
In contrast to the toric case,
the rank of $H^*(F;\Lambda)$ is greater
in general
than the rank of the Jacobi ring $\Jac(\po)$,
and hence than the number of Lagrangian torus fibers
with non-trivial Floer cohomology.
In the case of the 3-dimensional flag manifold $\Fl(3)$,
the potential function has six critical points,
which is equal to the rank of $H^*(\Fl(3);\Lambda)$.
Similarly, the potential function
for the Grassmannian $\Gr(2,5)$ of 2-planes in $\bC^5$
has ten critical points,
which is equal to the rank of $H^*(\Gr(2,5);\Lambda)$.
On the other hand, the number of critical points of the potential function
for the Grassmannian $\Gr(2,4)$ of 2-planes in $\bC^4$ is four,
which is less than the rank of $H^*(\Gr(2,4);\Lambda)$,
which is six.
%Eguchi, Hori, and Xiong \cite{Eguchi-Hori-Xiong_GQC} and 
%Rietsch \cite{MR2397456}
%considered a partial compactification of the algebraic torus
%to obtain as many critical points
%of the superpotential as the rank of $H^*(F;\bZ)$.

In this paper,
we study non-torus Lagrangian fibers of the Gelfand-Cetlin system
over the boundary of the Gelfand-Cetlin polytope
in the cases of %the 3-dimensional flag manifold 
$\Fl(3)$
and %the Grassmannian 
$\Gr(2,4)$. % of 2-planes in $\bC^4$.
The main results are the following:

\begin{theorem} \label{th:main1}
Let $\Phi \colon \Fl(3) \to \bR^3$ be the Gelfand-Cetlin system
with  the Gelfand-Cetlin polytope $\Delta = \Phi (\Fl(3))$.
\begin{enumerate}
 \item
 There exists a vertex $\bsu_0$ of $\Delta$ such that 
 a fiber $L(\bsu) = \Phi^{-1}(\bsu)$ over a boundary point
 $\bsu \in \partial \Delta$ 
 is a Lagrangian submanifold
 if and only if $\bsu = \bsu_0$. 
 \item
 The Lagrangian fiber $L(\bsu_0)$ is diffeomorphic to $\SU(2) \cong S^3$.
 \item
The Floer cohomology of $L(\bsu_0)$
over the Novikov field $\Lambda$ is trivial;
\begin{equation}
 \HF(L(\bsu_0), L(\bsu_0); \Lambda) = 0.
\end{equation}
\end{enumerate}
\end{theorem}

\begin{theorem} \label{th:main2}
Let $\Phi \colon \Gr(2,4) \to \bR^4$ be the Gelfand-Cetlin system 
with  the Gelfand-Cetlin polytope $\Delta = \Phi (\Gr(2,4))$.
\begin{enumerate}
 \item
 There exists an edge of $\Delta$ such that
 a fiber $L(\bsu) = \Phi^{-1}(\bsu)$ over $\bsu \in \partial \Delta$
 is a Lagrangian submanifold
 if and only if $\bsu$ is in the relative interior
 of the edge.
 \item
The Lagrangian fiber $L(\bsu)$ over any point $\bsu$
in the relative interior of the edge
is diffeomorphic to $\U(2) \cong S^1 \times S^3$.
 \item
$H^1(L(\bsu);\Lambda_0)$ is contained in $\scMwhat (L(\bsu))$.
 \item
The potential function is identically zero on $H^1(L(\bsu);\Lambda_0)$.
 \item
The Floer cohomology $\HF((L(\bsu), b), (L(\bsu), b); \Lambda)$
of a Lagrangian $\U(2)$-fiber $L(\bsu)$
over the Novikov field $\Lambda$ is non-trivial
if and only if $\bsu$ is the barycenter $\bsu_0$ of the edge
and $b = \pm \pi \sqrt{-1}/2 \, \bfe_1$,
where $\bfe_1$ is a generator of $H^1(L(\bsu); \bZ) \cong \bZ$.
 \item
 If the deformed Floer cohomology group
over the Novikov field
is non-trivial,
then it is isomorphic to the classical cohomology group;
\begin{align} \label{eq:HF1}
 \HF((L(\bsu_0),\pm \pi \sqrt{-1}/2 \, \bfe_1),
  (L(\bsu_0),\pm \pi \sqrt{-1}/2 \, \bfe_1);\Lambda)
  \cong H^*(S^1 \times S^3;\Lambda).
\end{align}
 \item
 The Floer cohomology of the pair
 $((L(\bsu_0), \pi \sqrt{-1}/2 \, \bfe_1),
 (L(\bsu_0), - \pi \sqrt{-1}/2 \, \bfe_1))$ is trivial;
\begin{align} \label{eq:HFpm}
 \HF((L(\bsu_0), \pi \sqrt{-1}/2 \, \bfe_1),
  (L(\bsu_0), \pi \sqrt{-1}/2 \, \bfe_1) ; \Lambda) = 0.
\end{align}
\end{enumerate}
\end{theorem}
More precise statements,
which describe the Floer cohomology groups
over the Novikov ring $\Lambda_0$,
are given in \pref{th:Fl(3)}, \pref{th:Gr(2,4)}, and \pref{th:Gr(2,4)_2}.
%
%In particular,
%there are as many as $\dim H^*(\Gr(2,4); \Lambda) = 6$
%balanced Lagrangian fibers of the Gelfand-Cetlin system
%on $\Gr(2,4)$.

A symplectic manifold $(X, \omega)$ is \emph{monotone}
if the cohomology class $[\omega]$ is positively proportional
to the first Chern class;
\begin{align}
 \exists \lambda > 0 \quad
 [\omega] = \lambda c_1(X).
\end{align}
The quantum cohomology ring of a monotone symplectic manifold
does not have any convergence issue,
and hence is defined over $\bC$.
A Lagrangian submanifold $L$ is \emph{monotone}
if the symplectic area of a disk bounded by $L$
is positively proportional to the Maslov index;
\begin{align}
 \exists \lambda > 0 \quad
 \forall \beta \in \pi_2(M,L) \quad
 \beta \cap \omega = \lambda \mu(\beta).
\end{align}
The $A_\infty$-operations
on the Lagrangian intersection Floer complex
of a monotone Lagrangian submanifold is defined over $\bC$.
The minimal Maslov number of oriented monotone Lagrangian submanifold
is greater than or equal to 2,
so that the obstruction class $\frakm_0(1)$
can be written as
$
 \frakm_0(1) = \frakm_0(L) \, \bfe_0,
$
where %$\bfe \in H^0(L)$ is the unit and
$\frakm_0(L) \in \bC$ is the count of Maslov index 2 disks
bounded by $L$,
weighted by their symplectic areas
and holonomies of a flat $U(1)$-bundle on $L$
along the boundaries of the disks.
The \emph{monotone Fukaya category}
is defined as the direct sum 
\begin{align}
 \cF(X) := \bigoplus_{\lambda \in \bC} \cF(X;\lambda),
\end{align}
where $\cF(X;\lambda)$ is an $A_\infty$-category over $\bC$
whose objects are monotone Lagrangian submanifolds,
equipped with flat $U(1)$-bundles,
satisfying $\frakm_0(L) = \lambda$.
For any monotone Lagrangian submanifold $L$,
there is a natural ring homomorphism
\begin{align}
 \QH(X) \to \HF(L,L),
\end{align}
which is known by Auroux \cite{Auroux_MSTD},
Kontsevich, and Seidel
to send $c_1(X) \in \QH(X)$
to $\frakm_0(1) \in \HF(L,L)$.
It follows that $\cF(X;\lambda)$ is trivial
unless $\lambda$ is an eigenvalue
of the quantum cup product
by $c_1(X)$.

Now consider the case when $X = \Gr(2,4)$,
which can be written as a quadric hypersurface
\begin{align}
 X = \lc [z_0:\cdots:z_5] \in \bP^5 \relmid
  z_0^2 = z_1^2 + \cdots + z_5^2 \rc.
\end{align}
The real locus $X_\bR$ is a monotone Lagrangian sphere,
which is the vanishing cycle along a degeneration into a nodal quadric
and split-generates the nilpotent summand $D^\pi \cF(X;0)$
of the monotone Fukaya category
\cite[Lemma 4.6]{MR2929086}.
The Floer cohomology $\HF(X_\bR,X_\bR)$ is semisimple,
and carries a formal $A_\infty$-structure
\cite[Lemma 4.7]{MR2929086}.
It follows that $D^\pi \cF(X;0)$ is equivalent
to the direct sum of two copies
of the derived category $D^b(\bC)$ of $\bC$-vector spaces.
On the other hand,
$(L(\bsu_0),\pm \pi \sqrt{-1}/2 \, \bfe_1)$ are also
objects of the nilpotent summand $D^\pi \cF(X;0)$
of the monotone Fukaya category,
which are non-zero by \eqref{eq:HF1}.
Since
$
 (L(\bsu_0),\pm \sqrt{-1}/2 \, \bfe_1)
$
is a pair of orthogonal non-zero objects
in a triangulated category
equivalent to $D^b(\bC) \oplus D^b(\bC)$,
they split-generate the whole category:

\begin{corollary}
The pair
$(L(\bsu_0),\pm \pi \sqrt{-1}/2 \, \bfe_1)$
%also 
split-generate $D^\pi \cF(\Gr(2,4) ; 0)$.
\end{corollary}

%This paper is organized as follows:
%In \pref{sc:GC_system}, we recall the construction of the Gelfand-Cetlin system, 
%and study non-torus Lagrangian fibers in the cases of the full flag manifold $Fl(3)$ 
%and the Grassmannians $\Gr(2,4)$, $\Gr(2,5)$.
%In \pref{sc:critical},
%we discuss critical points of the potential function 
%and eigenvalues of the quantum cup product
%by the first Chern class.
%Main theorems are proved in \pref{sc:HF_nontorus}.

\emph{Acknowledgment}:
We thank Hiroshi Ohta, Kaoru Ono, and Yoshihiro Ohnita 
for useful conversations.
Y.~N. is supported
by Grant-in-Aid for Young Scientists (No.23740055).
K.~U. is supported
by Grant-in-Aid for Young Scientists (No.24740043).

%%%%%%%%%%%%%%%%%%%%%%%%%%%%%%%%%%%%%%%%%%%%%%%%%%%%%%%%%%%%%%%%%%%%%%%%%%%%%%%
\section{Non-torus fibers of the Gelfand-Cetlin system}\label{sc:GC_system}

\subsection{Flag manifolds}

For a sequence $0 = n_0 < n_1 < \dots < n_r < n_{r+1} =n$ of integers,
let  $F  = F(n_1 , \dots , n_r , n)$ be the  flag manifold
consisting of flags
\[
  0 \subset V_1 \subset \dots \subset V_r \subset \mathbb{C}^n,
  \quad \dim V_i = n_i
\]
of $\bC^n$.
We write the full flag manifold and the Grassmannian as
$\Fl(n) = F(1,2, \dots, n)$ and $\Gr(k,n) = F(k,n)$
respectively.
The complex dimension of $F(n_1, \dots , n_r , n)$ is given by
\[
  N= N(n_1, \dots , n_r, n) := 
  \dim_{\mathbb{C}} F(n_1, \dots , n_r, n) = 
  \sum_{i=1}^r (n_i - n_{i-1})(n - n_i ).
\]
Let $P = P(n_1, \dots , n_r, n) \subset \GL(n, \mathbb{C})$ 
be the stabilizer subgroup of the standard flag 
$
 (V_i = \langle e_1, \dots , e_{n_i} \rangle)_{i=1}^r,
$
where $\{ e_i \}_{i=1}^n$ is the standard basis of $\bC^n$.
The intersection of $P$ and $\U(n)$ is
$\U(k_1) \times \dots \times \U(k_{r+1})$ for
$k_i = n_i - n_{i-1}$, and
$F$ is written as
\[
  F = \GL(n,\mathbb{C}) / P
    = \U(n) / (\U(k_1) \times \dots \times \U(k_{r+1}) ).
\]

We take a $\U(n)$-invariant inner product 
$\langle x, y \rangle = \tr xy^*$ 
on the Lie algebra $\fraku (n)$ of $\U(n)$,
and identify the dual vector space $\fraku (n)^*$ of $\fraku(n)$ 
with the space $\sqrt{-1} \fraku (n)$ of Hermitian matrices.
For $\bslambda = \mathrm{diag}\, ( \lambda_1 , \dots , \lambda_n) 
\in \sqrt{-1} \fraku (n)$ with
\begin{equation}
   \underbrace{\lambda_1 = \dots = \lambda_{n_1}}_{k_1}
         > \underbrace{\lambda_{n_1 +1} = \dots = \lambda_{n_2}}_{k_2}
         > \dots > 
        \underbrace{\lambda_{n_r +1} = \dots = \lambda_n}_{k_{r+1}},
   \label{eq:lambda}
\end{equation}
the flag manifold $F$ is identified with the adjoint orbit $\scO_{\bslambda}
\subset \sqrt{-1} \fraku (n)$ of $\bslambda$.
Note that $\scO_{\bslambda}$ consists of Hermitian matrices with fixed
eigenvalues $\lambda_1, \dots, \lambda_n$.
Let 
\[
  \omega (\ad_{\xi}(x), \ad_{\eta}(x)) 
  = \frac{1}{2\pi} \langle x, [\xi, \eta] \rangle,
  \quad \xi, \eta \in \fraku(n)
\]
be the (normalized) Kostant-Kirillov form on $\scO_{\bslambda}$.

For each $i = 1 , \dots , r$, we set
$\bP_i := \bP \bigl( \bigwedge^{n_i} \bC^n \bigr)
 \cong \bP^{\binom {n}{n_i} -1}$.
Then the Pl\"ucker embedding is given by
\[
  \iota : F \hookrightarrow 
  \prod_{i=1}^{r} \mathbb{P}_i,
  \quad 
  (0 \subset V_1 \subset \dots \subset V_r \subset
         \mathbb{C}^n) \mapsto 
  (\textstyle{\bigwedge^{n_1}} V_1, \dots ,
   \textstyle{\bigwedge^{n_r}} V_r).
\]
Let $\omega_{\bP_i}$ be the Fubini-Study form on $\bP_i$ 
normalized in such a way that it represents the first Chern class 
$c_1(\scO(1))$ of the hyperplane bundle.
Then the Kostant-Kirillov form $\omega$
and the first Chern form $c_1(F)$ of $F$ 
are given by
\begin{align*}
 \omega
  &= \sum_{i=1}^r (\lambda_{n_i} - \lambda_{n_{i+1}}) \omega_{\bP_i}
\end{align*}
and
\begin{align*}
 c_1(F)
  &= \sum_{i=1}^r  (n_{i+1} - n_{i-1} )\omega_{\bP_i}
\end{align*}
respectively.

\begin{example}
The 3-dimensional full flag manifold  $\Fl(3)$ is embedded into 
\[
  \bP_1 \times \bP_2 
  = \bP (\bC^3) \times \bP (\textstyle{\bigwedge^2 \bC^3}) 
  \cong \bP^2 \times \bP^2
\]
as a hypersurface.
The image of $\Fl(3)$ is given by the Pl\"ucker relation
\[
  Z_1 Z_{23} + Z_2 Z_{31} + Z_3 Z_{12} = 0,
\]
where $[Z_1:Z_2:Z_3]$ and $[Z_{23}:Z_{31}:Z_{12}]$
are the Pl\"ucker coordinates 
on $\bP_1$ and $\bP_2$ respectively.
\end{example}

\begin{example}
The Grassmannian $\Gr(2,4)$ of 2-plans in $\bC^4$ is embedded into
$\bP(\bigwedge^2 \bC^4) \cong \bP^5$ as a hypersurface.
The Pl\"ucker relation is given by
\[
  Z_{12}Z_{34} - Z_{13}Z_{24} + Z_{14}Z_{23} = 0,
\]
where $[Z_{12} : Z_{13} : Z_{14} : Z_{23} : Z_{24} : Z_{34}]$ is the
Pl\"ucker coordinates.
\end{example}

%%%%%%%%%%%%%%%%%%%%%%%%%%%%%%%%%%%%%%%%%%%%%%%%%%%%%%%%%%%%%%%%%%%%%%%%%%%%%%%%%%%%
\subsection{The Gelfand-Cetlin system}

For $x \in \mathcal{O}_{\bslambda}$ and $k = 1, \dots ,n-1$, 
let $x^{(k)}$ denote the upper-left $k \times k$ submatrix of $x$.
Since $x^{(k)}$ is also a Hermitian matrix, it has real eigenvalues
$\lambda^{(k)}_1(x) \ge \lambda^{(k)}_2(x) \ge \dots 
\ge \lambda^{(k)}_k(x)$.
By taking the eigenvalues for all $k = 1, \dots ,n-1$,
we obtain a set
$( \lambda^{(k)}_i )_{1 \le i \le k \le n-1}$
of $n(n-1)/2$ functions,
which satisfy the inequalities
\begin{equation}
%\begin{aligned}
\begin{alignedat}{17}
  \lmd 1 &&&& \lmd 2 &&&& \lmd 3 && \cdots && \lmd {n-1} &&&& \lmd n  \\
  & \uge && \dge && \uge && \dge &&&&&& \uge && \dge & \\
  && \llmd {n-1}1 &&&& \llmd {n-1}2 &&&&&&&& \llmd{n-1}{n-1} && \\
  &&& \uge && \dge &&&&&&&& \dge &&& \\
  &&&& \llmd {n-2}1 &&&&&&&& \llmd{n-2}{n-2} &&&& \\
  &&&&& \uge &&&&&& \dge &&&&& \\
  &&&&&& \dndots &&&& \updots &&&&&& \\
  &&&&&&& \uge && \dge &&&&&&& \\
  &&&&&&&& \llmd 11 &&&&&&&&& 
\end{alignedat}.
%\end{aligned}
\label{GC-pattern}
\end{equation}
It follows that
the number of non-constant $\lambda^{(k)}_i$ 
coincides with $N = \dim_{\bC} F$.
Let $I = I(n_1, \dots, n_r, n)$ denotes the set of pairs $(i,k)$ such that 
$\lambda_i^{(k)}$ is non-constant.
Then the {\em Gelfand-Cetlin system} is defined by
\[
  \Phi = ( \lambda^{(k)}_i )_{(i,k) \in I}  : F(n_1, \dots, n_r,n)
  \longrightarrow 
  \mathbb{R}^{N(n_1, \dots, n_r,n)}.
\]

\begin{proposition}[Guillemin and Sternberg \cite{Guillemin-Sternberg_GCS}]
  The map $\Phi$ is a completely integrable system
  on $(F(n_1, \dots, n_r,n), \omega)$. 
  The functions $\lambda_i^{(k)}$ are action variables, 
  and the image $\Delta = \Phi(F)$ is a convex polytope
  defined by \eqref{GC-pattern}.
  The fiber $L(\bsu) = \Phi^{-1}(\bsu)$ over each interior point 
  $\bsu \in \Int \Delta$ is a Lagrangian torus.
\end{proposition}

The image $\Delta \subset \bR^{N(n_1, \dots, n_r,n)}$ 
is called the {\em Gelfand-Cetlin polytope}.
The Gelfand-Cetlin system is not smooth on the locus where
$\lambda_k^{(i)} = \lambda_k^{(i+1)}$ for some $(i,k)$, 
or equivalently, where the Gelfand-Cetlin pattern \eqref{GC-pattern}
contains a set of equalities of the form
\[
  \begin{alignedat}{5}
    && \gcbox{\lambda_{k+1}^{(i+1)}} && \\
    & \rueq && \lueq && \phantom{\uge} &\\
    \gcbox{\lambda_k^{(i)}} &&&& \gcbox{\lambda_k^{(i+1)}} && \\
    & \lueq && \rueq &&& \\
    && \gcbox{\lambda_{k-1}^{(i)}} &&&& 
  \end{alignedat}.
\]
The image of such loci are faces of $\Delta$
of codimension greater than two
where $\Delta$ does not satisfy the Delzant condition.
Away from such faces, each fiber $\Phi^{-1}(\bsu)$ of $\Phi$ is 
an isotropic torus whose dimension is that of the face of $\Delta$ 
containing $\bsu$ in its relative interior.

%%%%%%%%%%%%%%%%%%%%%%%%%%%%%%%%%%%%%%%%%%%%%%%%%%%%%%%%%%%%%%%%%%%%%%%%%%%%%%%%%%%

\subsection{The case of $\Fl(3)$} \label{sc:fiber_Fl(3)}

\begin{figure}
     \centering
      \includegraphics*{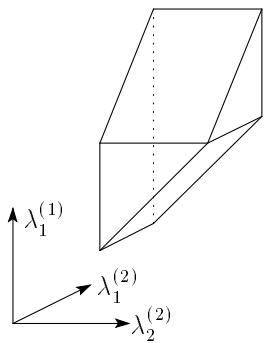}
     \caption{The Gelfand-Cetlin polytope for $\Fl(3)$}
     \label{fg:polytope_Fl(3)}
\end{figure}

After a translation by a scalar matrix, we may assume that
$\Fl (3)$ is identified with the adjoint orbit
of $\bslambda = \diag (\lambda_1, 0, -\lambda_2)$
for $\lambda_1, \lambda_2 > 0$.
Then the  Gelfand-Cetlin polytope $\Delta$ consists of 
$(u_1, u_2, u_3) \in \bR^3$ satisfying
\begin{equation}
  \begin{alignedat}{9}
    \gcbox{\lambda_1} &&&& \gcbox{0} &&&& \gcbox{-\lambda_2} \\
    & \uge && \dge && \uge && \dge &\\
    && \gcbox{u_1} &&&& \gcbox{u_2} && \\
    &&& \uge && \dge &&& \\
    &&&& \gcbox{u_3} &&&& 
  \end{alignedat}
  \label{eq:GC-pattern_Fl(3)}
\end{equation}
as shown in \pref{fg:polytope_Fl(3)}.
The non-smooth locus of $\Phi$ is the fiber $L_0 = \Phi^{-1}(\bszero)$
over the vertex $\bszero = (0,0,0) \in \Delta$
where four edges intersect.

\begin{definition}[Evans and Lekili {\cite[Definition 1.1.1]{1401.4073}}]
Let $K$ be a compact connected Lie group.
A Lagrangian submanifold $L$ in a K\"ahler manifold $X$ is 
said to be {\em $K$-homogeneous} if $K$ acts holomorphically on $X$ 
in such a way that $L$ is a $K$-orbit.
\end{definition}

\begin{proposition} \label{pr:SU(2)-fiber}
The fiber $L_0 = \Phi^{-1}(\bszero)$ is a Lagrangian 3-sphere
given by
\begin{equation*}
  L_0 = \left\{ \left.
    \begin{pmatrix} 
      0 & 0 & z_1 \\
      0 & 0 & z_2 \\
      \overline{z}_1 & \overline{z}_2 & \lambda_1 - \lambda_2
    \end{pmatrix}
    \in \sqrt{-1} \fraku (3) \right|
    |z_1|^2 + |z_2|^2 = \lambda_1 \lambda_2 \right\},
%  \label{eq:sphere1}
\end{equation*}
%The Lagrangian 3-sphere $L_0$
which is $K$-homogeneous for
\[
  K= 
  \left\{ \left.
    \begin{pmatrix} a_1 & -\overline{a}_2  & 0\\
    a_2 & \overline{a}_1 & 0 \\
    0 & 0 & 1 \end{pmatrix}
    \right|
    |a_1|^2+|a_2|^2=1
  \right\}
  \cong \SU(2).
\]
\end{proposition}

\begin{proof}
Suppose that $x \in L_0$.
Then $\lambda_1^{(2)}(x) = \lambda_2^{(2)}(x) = 0$
implies that $x^{(2)} =0$ and thus $x$ has the form
\[
  x=
  \begin{pmatrix} 
      0 & 0 & z_1 \\
      0 & 0 & z_2 \\
      \overline{z}_1 & \overline{z}_2 & x_{33}
    \end{pmatrix}
\]
for some $z_1, z_2 \in \bC$ and $x_{33} \in \bR$.
Since
\[ 
  \det (\lambda - x)
  = \lambda \lb \lambda^2 - x_{33} \lambda - (|z_1|^2+|z_2|^2) \rb
  =0
\]
has solutions $\lambda = \lambda_1, 0, -\lambda_2$, 
we have
$x_{33} = \lambda_1 - \lambda_2$ and 
$|z_1|^2 + |z_2|^2 = \lambda_1 \lambda_2$.
Hence the fiber $L_0$ is the $K$-orbit of 
\[
  \begin{pmatrix} 
      0 & 0 & \sqrt{\lambda_1 \lambda_2} \\ 
      0 & 0 & 0 \\
      \sqrt{\lambda_1 \lambda_2} & 0 & \lambda_1 - \lambda_2
  \end{pmatrix} 
  = 
  \Ad_{g_0}
  \begin{pmatrix}
    \lambda_1 & 0 & 0 \\ 
    0 & 0 & 0 \\ 
    0 & 0 & -\lambda_2
  \end{pmatrix}
  \in \scO_{\bslambda},
\]
where 
\[
  g_0 =
    \begin{pmatrix}
    \sqrt{\lambda_2/(\lambda_1 + \lambda_2)} & 0 
      & -\sqrt{\lambda_1/(\lambda_1 + \lambda_2)}\\
    0 & 1 & 0 \\
    \sqrt{\lambda_1/(\lambda_1 + \lambda_2)} & 0 
      & \sqrt{\lambda_2/(\lambda_1 + \lambda_2)}
  \end{pmatrix}
  \in \SU(3).
\]

Next we see that $L_0$ is Lagrangian.
Since $K$ acts transitively on $L_0$, the tangent space $T_xL_0$ is 
spanned by infinitesimal actions $\ad_{\xi}(x)$ of $\xi \in \frakk$,
where 
\[
  \frakk = \biggl\{ \xi = 
    \begin{pmatrix} \xi^{(2)} & 0 \\ 0 & 0 \end{pmatrix}
    \in \fraku(3)
    \biggm| \xi^{(2)} \in \mathfrak{su}(2) \biggr\}
  \cong \mathfrak{su}(2)
\]
is the Lie algebra of $K$.
Since $x^{(2)} = 0$ for $x \in L_0$, we have
\[
 \omega(\ad_{\xi}(x), \ad_{\eta}(x)) 
  = \frac{\ii}{2\pi} \tr \Bigl( x^{(2)}[\xi^{(2)}, \eta^{(2)}] \Bigr)
  = 0
\]
for any $\xi, \eta \in \frakk$.
\end{proof}

Let $\iota : \Fl(3) \to \bP_1 \times \bP_2 
= \bP (\bC^3) \times \bP (\bigwedge^2 \bC^3)$
be the Pl\"ucker embedding and 
$([Z_1:Z_2:Z_3], [Z_{23}:Z_{31}:Z_{12}])$ be the Pl\"ucker coordinates.
The Kostant-Kirillov form is given by
\[
  \omega = 
  \lambda_1 \omega_{\bP_1} + \lambda_2 \omega_{\bP_2}.
\]
Since the Lagrangian fiber $L_0$
as a submanifold in $\SU(3)/T$
consists of 
\[
  \begin{pmatrix} a_1 & -\overline{a}_2  & 0\\
    a_2 & \overline{a}_1 & 0 \\
    0 & 0 & 1 
  \end{pmatrix}
  g_0
  = \frac 1{\sqrt{\lambda_1+\lambda_2}}
  \begin{pmatrix}
    \sqrt{\lambda_2} a_1 
      & -\sqrt{\lambda_1+\lambda_2} \overline{a}_2 
      & -\sqrt{\lambda_1} a_1\\
    \sqrt{\lambda_2} a_2 
      & \sqrt{\lambda_1+\lambda_2} \overline{a}_1 
      & -\sqrt{\lambda_1} a_2\\
    \sqrt{\lambda_1} & 0 & \sqrt{\lambda_2}
  \end{pmatrix}
  \mod T
\]
with $|a_1|^2+|a_2|^2=1$,
the image $\iota(L_0)$ is given by
\begin{equation}
  \iota (L_0) = 
  \Biggl\{
  \Biggm (\left[ a_1:a_2: \sqrt{\frac{\lambda_1}{\lambda_2}} \right], 
   \left[\overline{a}_1: \overline{a}_2: 
     -\sqrt{\frac{\lambda_2}{\lambda_1}} \right] \Biggm)
    % \in \bP_1 \times \bP_2
   \, \Biggm|\, |a_1|^2+|a_2|^2=1 \Biggr\}.
  \label{eq:lagS^3}
\end{equation}
Define an anti-holomorphic involution 
$\tau$ on $\Fl(3)$ by
\begin{align}
 \tau \lb [Z_1:Z_2:Z_3], [Z_{23}:Z_{31}:Z_{12}] \rb
  &= \lb
         \ld
          \Zbar_{23}:\Zbar_{31}:
           - \frac{\lambda_1}{\lambda_2} \Zbar_{12}
         \rd,
         \ld
          \overline{Z}_1:\overline{Z}_2:
             - \frac{\lambda_2}{\lambda_1} \overline{Z}_3
         \rd
        \rb.
  \label{eq:involution_Fl(3)}
\end{align}

\begin{proposition}
The Lagrangian $L_0$ is the fixed point set of $\tau$.
\end{proposition}

One can easily see that $\tau$ is an anti-symplectic involution 
if and only if $\lambda_1 = \lambda_2$.

%%%%%%%%%%%%%%%%%%%%%%%%%%%%%%%%%%%%%%%%%%%%%%%%%%%%%%%%%%%%%%%%%%%%%%%%%%%%%%

\subsection{The case of $\Gr(2,4)$} \label{sc:fiber_Gr(2,4)}

For $k < n$, let $\Vtilde (k,n)$ be the space of $n \times k$ matrices of rank $k$, and set
\[
  V(k,n) = \{ Z \in \Vtilde (k,n) \mid Z^* Z = I_k \}.
\]
Then the Grassmannian $\Gr(k,n)$ is given by
\[
  \Gr(k,n) = \Vtilde (k,n) / \GL(k, \bC) = V(k,n) / \U(k).
\]

We first consider the Gelfand-Cetlin system on $\Gr(n, 2n)$ for general $n$.
Fix $\lambda >0$ and identify $\Gr(n,2n)$ with the adjoint orbit 
$\scO_{\bslambda}$ of
\begin{align*}
  \bslambda &= 
  \diag(\underbrace{\lambda, \dots, \lambda}_n, 
                    \underbrace{-\lambda, \dots, -\lambda}_n).
%  &= 2\lambda \begin{pmatrix} I_n & \\ & 0 \end{pmatrix} 
%     - \lambda \begin{pmatrix} I_n & \\ & I_n \end{pmatrix}.
\end{align*}
The orbit $\scO_{\bslambda}$ consists of matrices of the form
$2 \lambda Z Z^* - \lambda I_{2n}$ for $Z \in V(n, 2n)$.
The  Gelfand-Cetlin polytope $\Delta$ of $\Gr(n,2n)$ consists of 
$\bsu = (u_i^{(k)})_{(i,k) \in I} \in \bR^{n^2}$ satisfying
\[
  \begin{alignedat}{13}
    &&&&&& \gcbox{u_n^{(2n-1)}} \\
    &&&&& \dge && \uge \\
    \gcbox{\lambda} &&&& \gcbox{\updots} &&&& \gcbox{\dndots} &&&& \gcbox{-\lambda} \\
    & \gcbox{\uge} && \gcbox{\dge} &&&&&& \gcbox{\uge} && \gcbox{\dge} \\
    && \gcbox{u_1^{(n)}} &&&& \gcbox{\cdots} &&&& \gcbox{u_n^{(n)}} \\
    &&& \gcbox{\uge} &&&&&& \gcbox{\dge} \\
    &&&& \gcbox{\dndots} &&&& \gcbox{\updots} \\
    &&&&& \gcbox{\uge} && \gcbox{\dge} \\
    &&&&&& \gcbox{u_1^{(1)}}
  \end{alignedat}.
\]
For $-\lambda < t < \lambda$, let $L_t = \Phi^{-1}(t, \dots, t)$ 
be the fiber over the boundary point
$u_1^{(1)} = \dots = u^{(2n-1)}_{n} = t$
of $\Delta$.

\begin{proposition}
The fiber $L_t$ is a Lagrangian submanifold given by
\[
  L_t = \biggl\{
  \begin{pmatrix}
    t I_n & \sqrt{\lambda^2 -t^2} A^* \\
    \sqrt{\lambda^2 -t^2} A & -t I_n 
  \end{pmatrix}
  \in \sqrt{-1} \fraku (2n)
  \biggm| A \in \U(n) \biggr\}
  \cong \U(n),
\]
%which is the orbit of
%\[
% \begin{pmatrix}
%      t I_n & \sqrt{\lambda^2 -t^2} I_n \\
%    \sqrt{\lambda^2 -t^2} I_n & -t I_n 
%  \end{pmatrix}
%  = \Ad_{g_t} 
%    \begin{pmatrix} \lambda I_n & \\ 
%         0 & -\lambda I_n \end{pmatrix} 
%  \in \scO_{\bslambda}
%\]
%under the adjoint action of 
which is $K$-homogeneous for
\[
  K = \left\{ \left. 
        \begin{pmatrix} P & 0 \\ 0 & I_n \end{pmatrix} 
        \in \U(2n) 
        \, \right| \, 
        P \in \U(n) \right\} \cong \U(n).
\]
%where
%\[ 
%  g_t = \frac {1}{\sqrt {2 \lambda}}
%        \begin{pmatrix}
%          \sqrt{\lambda +t} I_n & -\sqrt{\lambda -t} I_n \\
%          \sqrt{\lambda -t} I_n & \sqrt{\lambda +t} I_n
%        \end{pmatrix}
%  \in \U(2n).
%\]
%the fiber $L_t$ is $K$-homogeneous.
\end{proposition}

\begin{proof}
We write $x \in \scO_{\bslambda}$ as
\[
  x= 2 \lambda Z Z^* - \lambda I_{2n}
   = \lambda 
     \begin{pmatrix}
       2 Z_1 Z_1^* - I_n & 2 Z_1 Z_2^* \\
       2 Z_2 Z_1^* & 2 Z_2 Z_2^* - I_n
     \end{pmatrix}	
\] 
for $n \times n$ matrices $Z_1$, $Z_2$ with
\[
  Z = \begin{pmatrix} Z_1 \\ Z_2 \end{pmatrix}
  \in V(n, 2n).
\]
Suppose that $x \in L_t$, or equivalently,
$\lambda_1^{(n)} (x) = \dots = \lambda_n^{(n)}(x) = t$.
Then the upper-left $n \times n$ block of $x$ satisfies
\[ 
  x^{(n)} = 2 \lambda Z_1 Z_1^* - \lambda I_n = tI_n,
\]
which means that $Z_1 \in \sqrt{(\lambda + t)/2\lambda} \U(n)$. 
After the right $\U(n)$-action on $V(n, 2n)$, 
we may assume that 
$Z_1 = \sqrt{(\lambda + t)/2\lambda} I_n$.
Then the condition $Z^* Z = I_n$ implies that
\[
  Z_2^* Z_2 = I_n - \frac{\lambda + t}{2 \lambda} I_n
  = \frac{\lambda - t}{2\lambda} I_n.
\]
Hence $Z$ has the form
\begin{equation}
  Z = \begin{pmatrix}
    \sqrt{(\lambda + t)/2 \lambda} I_n \\
    \sqrt{(\lambda - t)/2 \lambda} A
  \end{pmatrix}
  \in V(n, 2n)
  \label{eq:U(n)fiber_in_V} 
\end{equation}
for some $A \in \U(n)$,
which shows that
\[
  x = 2 \lambda Z Z^* - \lambda I_{2n} =
  \begin{pmatrix}
    t I_n & \sqrt{\lambda^2 -t^2} A^* \\
    \sqrt{\lambda^2 -t^2} A & -t I_n 
  \end{pmatrix}.
\]
The $K$-homogeneity is obvious from this expression.
Since 
the tangent space $T_xL_t$ is spanned by the infinitesimal action of 
the Lie algebra $\frakk$ of $K$, 
and $x^{(n)} = t I_n$ is a scalar matrix, 
we have 
\[
  \omega_x (\ad_{\xi}(x), \ad_{\eta}(x) ) 
  = \frac 1{2\pi} \tr x^{(n)} [\xi^{(n)}, \eta^{(n)}] = 0
\]
for 
\[
  \xi = \begin{pmatrix} \xi^{(n)} & \\ & 0 \end{pmatrix}, \ 
  \eta = \begin{pmatrix} \eta^{(n)} & \\ & 0 \end{pmatrix}
  \in \frakk,
\]
which shows that $L_t$ is Lagrangian.
\end{proof}

\begin{corollary} \label{cr:dislpaceable_Gr(2,4)}
For $t \ne 0$, the fiber $L_t$ is displaceable, 
i.e., there exists a Hamiltonian diffeomorphism $\varphi$ on $\Gr(n,2n)$ 
such that $\varphi (L_t) \cap L_t = \emptyset$.
\end{corollary}

\begin{proof}
One has $g (L_t) = L_{-t}$ for
$
  g = 
  \begin{pmatrix} 0 & -I_n \\ I_n & 0 \end{pmatrix}
  \in \U(2n).
$
\end{proof}

%\begin{corollary}
%  If $t \ne 0$, then $HF((L_t, b), (L_t, b); \Lambda) = 0$
%  for $b \in H^1(L_t)$.
%\end{corollary}

In the rest of this subsection,  
we restrict ourselves to the case of $\Gr(2,4)$.
We write $(u_1, u_2, u_3, u_4) = (u_2^{(3)}, u_1^{(2)}, u_2^{(2)}, u_1^{(1)})$
for simplicity.
Figure \ref{fg:polytope(2,4)} shows the projection 
\[
  \Delta \longrightarrow [-\lambda, \lambda], \quad 
  \bsu=(u_1, u_2, u_3, u_4) \longmapsto u_1.
\]
\begin{figure}
  \centering
    \includegraphics*{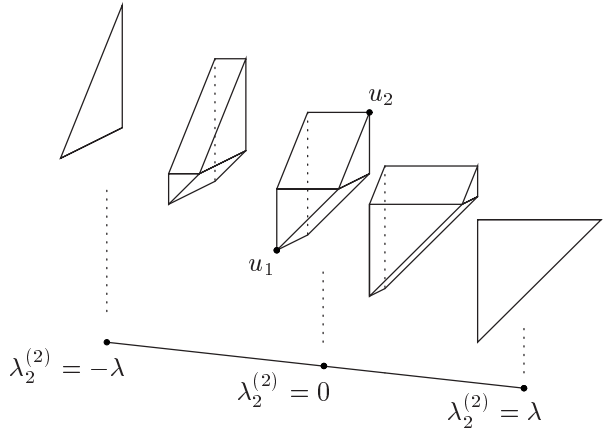}
  \caption{The Gelfand-Cetlin polytope for $\Gr(2,4)$}
  \label{fg:polytope(2,4)}
\end{figure}
The non-smooth locus of $\Phi$ is the inverse image of the edge of $\Delta$
defined by $u_1= \dots = u_4$.
The fiber $L_t$ over $(t,t,t,t) \in \partial \Delta$ is a Lagrangian submanifold
consists of
$2\lambda ZZ^* - \lambda I_{2n}$ with
\[
  Z = \frac {1}{\sqrt {2\lambda}}
        \begin{pmatrix}
          \sqrt{\lambda +t} I_2  \\
          \sqrt{\lambda -t} A 
        \end{pmatrix}
  \mod \U(2) 
\] 
for $A \in \U(2)$.
We identify $\U(2)$ with $\U(1) \times \SU(2) \cong S^1 \times S^3$
by
\[
  \U(1) \times \SU(2) \longrightarrow \U(2),
  \quad \left( a_0, 
  \begin{pmatrix} a_1 & -\overline{a}_2 \\ 
    a_2 & \overline{a}_1 \end{pmatrix}
  \right)
  \longmapsto
  \begin{pmatrix} a_0 & 0 \\ 0 & 1
  \end{pmatrix}
  \begin{pmatrix} a_1 & - \overline{a}_2 \\ 
    a_2 & \overline{a}_1 \end{pmatrix}.
\]
Then the image of $L_t$ under the Pl\"ucker embedding 
$\iota: \Gr(2,4) \to \bP (\bigwedge^2 \bC^4) \cong \bP^5$ is given by
\[
  \iota (L_t) = 
  \Biggl\{ \biggl[ \sqrt{\frac{\lambda +t}{\lambda -t}} : 
  - a_0 \overline{a}_2 : \overline{a}_1 : - a_0 a_1: - a_2: 
  \sqrt{\frac{\lambda -t}{\lambda +t}} a_0
  \biggr] 
  \Biggm| |a_0|^2 = |a_1|^2+|a_2|^2 =  1 \Biggr\}.
\]

This expression implies the following.

\begin{proposition}
For each $t \in (-\lambda, \lambda)$, 
we define an anti-holomorphic involution $\tau_t$
on $\Gr(2,4)$ defined by
\begin{equation}
  \tau_t([Z_{12}: Z_{13}: Z_{14}: Z_{23}: Z_{24}: Z_{34}]) 
  = \biggl[ \frac{\lambda +t}{\lambda -t} \overline{Z}_{34}: 
     \overline{Z}_{24}: -\overline{Z}_{23}:
     -\overline{Z}_{14}: \overline{Z}_{13}: 
     \frac{\lambda -t}{\lambda +t}\overline{Z}_{12} \biggr]
  \label{eq:involution_Gr(2,4)}
\end{equation}
Then $L_t$ is the fixed point set of $\tau_t$.
\end{proposition}

\begin{remark}
The map $\tau_0$ for $t=0$ is an anti-symplectic involution as well, 
and satisfies $\tau_0(L_t) = L_{-t}$ for each $t \in (-\lambda, \lambda)$.
\end{remark}

%%%%%%%%%%%%%%%%%%%%%%%%%%%%%%%%%%%%%%%%%%%%%%%%%%%%%%%%%%%%%%%%%%%%%%%%%%%%
\subsection{The case of $\Gr (2,5)$} \label{sc:fiber_Gr(2,5)}

We fix $\lambda >0$ and identify $\Gr(2,5)$ with the 
adjoint orbit $\scO_{\bslambda}$ of 
$\diag (\lambda, \lambda, 0,0,0) \in \ii \fraku (5)$.
The Gelfand-Cetlin polytope $\Delta$ is defined by
\begin{equation}
  \begin{alignedat}{11}
    \gcbox{\lambda} &&&& \gcbox{u_1}  \\
    & \uge && \dge && \uge \\
    && \gcbox{u_2} &&&& \gcbox{u_3} &&&& \gcbox{0} \\
    &&& \uge && \dge && \uge && \dge \\
    &&&& \gcbox{u_4} &&&& \gcbox{u_5}  \\
    &&&&& \uge && \dge \\
    &&&&&& \gcbox{u_6}
\end{alignedat}
\label{eq:GCpattern_Gr(2,5)}
\end{equation}

We first consider the fiber $L_1(s_1, s_2, t)$ over
a boundary point given by
\[
  \begin{alignedat}{11}
    \gcbox{\lambda} &&&& \gcbox{s_2}  \\
    & \ugne && \dgne && \ugne \\
    && \gcbox{s_1} &&&& \gcbox{t} &&&& \gcbox{0} \\
    &&& \ugne && \deq && \ueq && \dgne \\
    &&&& \gcbox{t} &&&& \gcbox{t}  \\
    &&&&& \ueq && \deq \\
    &&&&&& \gcbox{t}
\end{alignedat}.
\]

\begin{proposition}
  The fiber $L_1(s_1, s_2, t)$ is a Lagrangian submanifold diffeomorphic 
  to $\U(2) \times T^2 \cong S^3 \times T^3$.
  Moreover, $L_1(s_1, s_2, t)$ is $K$-homogeneous for
  \[
    K = \left\{ \left.
    \begin{pmatrix} 
      P \\ 
      & e^{\sqrt{-1}\theta_1} \\
      && e^{\sqrt{-1}\theta_2} \\
      &&& 1
    \end{pmatrix} \in \U(5)
    \right|
    P \in \U(2), \, \theta_1, \theta_2 \in \bR
    \right\}
    \cong 
    \U(2) \times T^2.
  \]
\end{proposition}

\begin{proof}
Note that $\scO_{\bslambda}$ consists of matrices of the form
\begin{align} \label{eq:251}
  x = \lambda Z Z^*
  = \lambda ( z_i \overline{z}_j + w_i \overline{w}_j)_{1 \le i, j \le 5}
\end{align}
for
\[
  Z = \begin{pmatrix} 
          z_1 & w_1 \\
          z_2 & w_2 \\
          z_3 & w_3 \\
          z_4 & w_4 \\
          z_5 & w_5 
      \end{pmatrix}
  \in V(2,5), 
\]
i.e., 
\begin{align} \label{eq:252}
  \sum_{i=1}^5 |z_i|^2 = \sum_{i=1}^5 |w_i|^2 = 1, \quad
  \sum_{i=1}^5 z_i \overline{w}_i = 0.
\end{align}
Since the upper-left $2 \times 2$ submatrix of 
$x = \lambda ( z_i \overline{z}_j + w_i  \overline{w}_j) 
\in L_1(s_1, s_2, t)$ satisfies
\begin{align} \label{eq:253}
  x^{(2)} = \lambda \begin{pmatrix}
    |z_1|^2 + |w_1|^2 & z_1 \overline{z}_2 + w_1  \overline{w}_2 \\
    z_2 \overline{z}_1 + w_2 \overline{w}_1 & |z_2|^2 + |w_2|^2
  \end{pmatrix}
  = \begin{pmatrix} t & 0 \\ 0 & t \end{pmatrix}, 
\end{align}
we have
\begin{align} \label{eq:254}
  \sqrt{\frac{\lambda}{t}}
  \begin{pmatrix} z_1 & w_1 \\ z_2 & w_2 \end{pmatrix}
  \in \U(2),
\end{align}
and in particular, 
$|z_1|^2 + |z_2|^2 = |w_1|^2 + |w_2|^2 = t/\lambda$.
Then the condition \eqref{eq:252} implies
\begin{align}
 |z_3|^2 + |z_4|^2 + |z_5|^2 &= (\lambda - t)/\lambda,
  \label{eq:cond_L1_1} \\
 |w_3|^2 + |w_4|^2 + |w_5|^2 &= (\lambda - t)/\lambda,
  \label{eq:cond_L1_2} \\
 z_3 \overline{w}_3 + z_4 \overline{w}_4 + z_5 \overline{w}_5 &= 0.
  \label{eq:cond_L1_3}
\end{align}
On the other hand, the conditions $\tr x^{(3)} = s_1 + t$,
$\tr x^{(4)} = \lambda + s_2$ imply
\begin{align}
 |z_3|^2 + |w_3|^2 &= (s_1 - t)/ \lambda,
  \label{eq:cond_L1_4} \\
 |z_4|^2 + |w_4|^2 &= (\lambda - s_1 + s_2 - t)/ \lambda,
  \label{eq:cond_L1_5} \\
 |z_5|^2 + |w_5|^2 &= (\lambda - s_2)/ \lambda.
  \label{eq:cond_L1_6}
\end{align}
After the right $\SU(2)$-action on $(z,w)$, we may assume that
$(z_5, w_5) = \bigl( \sqrt{(\lambda - s_2)/ \lambda} , 0 \bigr)$. 
Then \eqref{eq:cond_L1_1}, \eqref{eq:cond_L1_2}, and \eqref{eq:cond_L1_3}
become 
\begin{align*}
 |z_3|^2 + |z_4|^2  &= (s_2 - t)/\lambda,
%  \label{eq:cond_L1_7} 
  \\
 |w_3|^2 + |w_4|^2  &= (\lambda - t)/\lambda,
%  \label{eq:cond_L1_8} 
  \\
 z_3 \overline{w}_3 + z_4 \overline{w}_4  &= 0,
%  \label{eq:cond_L1_9}
\end{align*}
which mean that the $2 \times 2$ submatrix $(z_i, w_i)_{i=3,4}$ has the form
\[
  \begin{pmatrix} z_3 & w_3 \\ z_4 & w_4 \end{pmatrix}
  =
  \begin{pmatrix}
    \sqrt{(s_2 - t)/\lambda} \, a 
      & -\sqrt{(\lambda - t)/\lambda} \, \overline{b} c \\
    \sqrt{(s_2 - t)/\lambda} \, b 
      & \sqrt{(\lambda - t)/\lambda} \, \overline{a} c
  \end{pmatrix}
\]
for some
\[
  \begin{pmatrix} a & - \overline{b} \\ b & \overline{a} \end{pmatrix}
  \in \SU(2),
  \quad
  c \in \U(1).
\]
Combining this with \eqref{eq:cond_L1_4} and \eqref{eq:cond_L1_5} we have
\[
  |a|^2 = \frac{\lambda - s_1}{\lambda - s_2},
  \quad
  |b|^2 = \frac{s_1 - s_2}{\lambda - s_2},
\]
and hence
\[
  \begin{pmatrix} z_3 & w_3 \\ z_4 & w_4 \end{pmatrix}
  = \frac{1}{\sqrt{\lambda (\lambda - s_2)}}
   \begin{pmatrix}
    \sqrt{(s_2 - t)(\lambda - s_1)} \, e^{\sqrt{-1} \theta_1} 
      & -\sqrt{(\lambda - t)(s_1 - s_2)} \, e^{- \sqrt{-1} \theta_2} c \\
    \sqrt{(s_2 - t)(s_1 - s_2)} \, e^{\sqrt{-1} \theta_2} 
      & \sqrt{(\lambda - t)(\lambda - s_1)} \, e^{- \sqrt{-1} \theta_1} c
  \end{pmatrix}
\]
for some $\theta_1, \theta_2 \in \bR$.
After the action of 
\[
  \left\{ \left.
  \begin{pmatrix} 1 & 0 \\ 0 & e^{\sqrt{-1} \varphi} \end{pmatrix}
  \in \U(2) \right|
  \varphi \in \bR 
  \right\}
\cong \U(1)
\]
from the right, we may assume that
\[
  \begin{pmatrix} z_3 & w_3 \\ z_4 & w_4 \end{pmatrix}
  = \frac{1}{\sqrt{\lambda (\lambda - s_2)}}
   \begin{pmatrix}
    \sqrt{(s_2 - t)(\lambda - s_1)} \, e^{\sqrt{-1} \theta_1} 
      & -\sqrt{(\lambda - t)(s_1 - s_2)} \, e^{\sqrt{-1} \theta_1} \\
    \sqrt{(s_2 - t)(s_1 - s_2)} \, e^{\sqrt{-1} \theta_2} 
      & \sqrt{(\lambda - t)(\lambda - s_1)} \, e^{ \sqrt{-1} \theta_2}
  \end{pmatrix}.
\]
Therefore $Z = (z_i, w_i)_i$ is normalized as
\[
  \begin{pmatrix} z_1 & w_1 \\ \vdots & \vdots \\ z_5 & w_5 \\
  \end{pmatrix}
  =
   \begin{pmatrix}
   z_1 & w_1 \\ 
   z_2 & w_2 \\
    \sqrt{(s_2 - t)(\lambda - s_1)/\lambda (\lambda - s_2)} \, e^{\sqrt{-1} \theta_1} 
      & -\sqrt{(\lambda - t)(s_1 - s_2)/\lambda (\lambda - s_2)} \, e^{\sqrt{-1} \theta_1} \\
    \sqrt{(s_2 - t)(s_1 - s_2)/\lambda (\lambda - s_2)} \, e^{\sqrt{-1} \theta_2} 
      & \sqrt{(\lambda - t)(\lambda - s_1) /\lambda(\lambda - s_2)} \, e^{ \sqrt{-1} \theta_2} \\
    \sqrt{(\lambda - s_2)/ \lambda} & 0
  \end{pmatrix}
\]
with \eqref{eq:254},
which implies that $L_1(s_1, s_2, t)$ is a $K$-orbit and diffeomorphic to 
$\U(2) \times T^2$.

The assertion that $L_1(s_1, s_2, t)$ is Lagrangian follows from 
the $K$-homogeneity as in the cases of $\Fl(3)$ and $\Gr(n,2n)$.
\end{proof}

Next we consider the fiber $L_2(s_1, s_2, t)$ over
\[
  \begin{alignedat}{11}
    \gcbox{\lambda} &&&& \gcbox{t}  \\
    & \ugne && \deq && \ueq \\
    && \gcbox{t} &&&& \gcbox{t} &&&& \gcbox{0} \\
    &&& \ueq && \deq && \ugne && \dgne \\
    &&&& \gcbox{t} &&&& \gcbox{s_1}  \\
    &&&&& \ugne && \dgne \\
    &&&&&& \gcbox{s_2}
\end{alignedat}.
\]
Suppose that 
$x = \lambda (z_i \overline{z}_j + w_i \overline{w}_j)_{1 \le i, j \le 5}
\in L_2(s_1, s_2, t)$.
The condition that 
$x^{(3)} = \lambda (z_i \overline{z}_j + w_i \overline{w}_j)_{1 \le i, j \le 3}$ 
has eigenvalues $t, t, 0$ is equivalent to
\begin{align}
  |z_1|^2 + |z_2|^2 + |z_3|^2 &= t/\lambda, \\
  |w_1|^2 + |w_2|^2 + |w_3|^2 &= t/\lambda, \\
  z_1 \overline{w}_1 + z_2 \overline{w}_2 + z_3 \overline{w}_3 &= 0,
\end{align}
and hence 
\[
  \sqrt{\frac{\lambda}{\lambda - t}}
  \begin{pmatrix}
    z_4 & w_4 \\ z_5 & w_5 
  \end{pmatrix}
  \in U(2).
\]
On the other hand, the conditions $x^{(1)} = s_2$, 
$\tr x^{(2)} = t + s_1$, and $\tr x^{(3)} = 2t$ imply
\begin{align*}
  |z_1|^2 + |w_1|^2 &= s_2 / \lambda, \\
  |z_2|^2 + |w_2|^2 &= (t - s_2 + s_1)/ \lambda, \\
  |z_3|^2 + |w_3|^2 &= (t - s_1)/ \lambda.
\end{align*}
Then we have the following.

\begin{proposition} \label{pr:dislpaceable_Gr(2,5)}
  The fiber $L_2(s_1, s_2, t)$ is a $\U(2) \times T^2$-homogeneous Lagrangian submanifold
  diffeomorphic to $\U(2) \times T^2 \cong S^3 \times T^3$.
  Moreover, the fibers $L_1(s_1, s_2, t)$ and $L_2(s_1, s_2, t)$ satisfy
  \[
    g(L_2(s_1, s_2, t)) = L_1(\lambda - s_1, \lambda - s_2, \lambda - t)
  \]
  for
  \[
    g = \begin{pmatrix}
         0 & & 1 \\ & \updots &  \\ 1 & & 0
        \end{pmatrix}
    \in \U(5).
  \]
  In particular, 
  $L_1(s_1, s_2, t)$ and $L_2(s_1, s_2, t)$ are displaceable.
\end{proposition}

The Hamiltonian isotopy invariance of the Floer cohomology 
over the Novikov field \cite[Theorem G]{Fukaya-Oh-Ohta-Ono}
implies the following.

\begin{corollary} \label{cr:Gr25}
For $i=1, 2$, we have
\[
   \HF((L_i(s_1, s_2, t), b), (L_i(s_1, s_2, t), b) ; \Lambda) = 0
\]
for any weak bounding cochain $b$.
\end{corollary}

\begin{remark}
  Other boundary fibers have lower dimensions.
  For example, the fiber over 
  \[
  \begin{alignedat}{11}
    \gcbox{\lambda} &&&& \gcbox{t}  \\
    & \ugne && \deq && \ueq \\
    && \gcbox{t} &&&& \gcbox{t} &&&& \gcbox{0} \\
    &&& \ueq && \deq && \ueq && \dgne \\
    &&&& \gcbox{t} &&&& \gcbox{t}  \\
    &&&&& \ueq && \deq \\
    &&&&&& \gcbox{t}
  \end{alignedat}
  \]
  consists of
  \[
    \begin{pmatrix}
      \sqrt{t/\lambda} & 0 \\ 
      0 & \sqrt{t/\lambda} \\
      0 & 0 \\
      z_4 & w_4 \\
      z_5 & w_5
    \end{pmatrix}
    \mod \U(2)
  \]
  with
  \[
    \begin{pmatrix} 
      z_4 & w_4 \\ z_5 & w_5
    \end{pmatrix}
    \in \sqrt{(\lambda - t)/\lambda} \U(2),
  \]
  which means that the fiber is diffeomorphic to $\U(2)$.
\end{remark}

%%%%%%%%%%%%%%%%%%%%%%%%%%%%%%%%%%%%%%%%%%%%%%%%%%%%%%%%%%%%%%%%%%%%%%%

\section{Critical points of the potential function}
 \label{sc:critical}

Let 
$
 \Phi : F = F(n_1, \dots, n_r, n) \to \Delta
$
be the Gelfand-Cetlin system on the flag manifold, and
$\{ \theta^{(k)}_i \}_{(i,k) \in I}$ be the angle variables 
dual to the action variables $\{ \lambda_i^{(k)} \}_{(i,k) \in I}$.
For each $\bsu = (u_k^{(i)})_{(i,k) \in I} \in \Int \Delta$, 
we identify $H^1(L(\bsu); \Lambda_0)$ with $\Lambda_0^N$ by
\[
  b = \sum_{(i,k) \in I} x^{(k)}_i d \theta^{(k)}_i 
  \in H^1(L(\bsu); \Lambda_0)
  \longleftrightarrow 
  \bsx = (x^{(k)}_i)_{(i,k) \in I} \in \Lambda_0^N,
\]
and set
\begin{align*}
 y_i^{(k)} &= e^{x_i^{(k)}} T^{u_i^{(k)}}, \qquad (i, k) \in I, \\
 Q_j &= T^{\lambda_{n_j}}, \qquad j = 1, \dots, r+1.
\end{align*}

\begin{theorem}[{\cite[Theorem 10.1]{Nishinou-Nohara-Ueda_TDGCSPF}}] \label{th:potential}
For any interior point $\bsu \in \Int \Delta$, 
we have an inclusion $
H^1(L(\bsu); \Lambda_0) \subset \scMwhat(L(\bsu))$.
As a function on
\[
 \bigcup_{\bsu \in \Int \Delta} H^1(L(\bsu); \Lambda_0) \cong 
 \Int \Delta \times \Lambda_0^N,
\]
the potential function is given by
\begin{equation*} %\label{eq:toric_potential}
 \po(\bsu, \bsx)
  = \sum_{(i, k) \in I}
     \left( \frac{y_i^{(k+1)}}{y_i^{(k)}} 
          + \frac{y_i^{(k)}}{y_{i+1}^{(k+1)}} \right),
\end{equation*}
where we put $y_i^{(k+1)} = Q_j$ 
if $\lambda_i^{(k+1)} = \lambda_{n_j}$ is a constant function.
\end{theorem}

\begin{example}
We identify the 3-dimensional flag manifold $\Fl(3)$
with the adjoint orbit of
$\bslambda = \diag(\lambda_1, \lambda_2, \lambda_3)$.
The potential function is given by
\begin{align*}
 \po
  &= e^{-x_1} T^{- u_1 + \lambda_1}
     + e^{x_1} T^{u_1 - \lambda_2}
     + e^{- x_2} T^{- u_2 + \lambda_2} \\
  & \qquad
     + e^{x_2} T^{u_2 - \lambda_3}
     + e^{x_1 - x_3} T^{u_1 - u_3}
     + e^{- x_2 + x_3} T^{- u_2 + u_3} \\
  &= \frac{Q_1}{y_1} + \frac{y_1}{Q_2} + \frac{Q_2}{y_2}
     + \frac{y_2}{Q_3} + \frac{y_1}{y_3} + \frac{y_3}{y_2}.
\end{align*}
The potential function $\po$ has six critical points given by
\begin{align*}
 y_1 &= {y_3^2}/{y_2}, \\ %\frac{y_3^2}{y_2}, \\
 y_2 &= \pm \sqrt{Q_3( y_3 + Q_2)}, \\
 y_3 &= \sqrt[3]{Q_1 Q_2 Q_3}, \ 
        e^{2 \pi \sqrt{-1}/3} \sqrt[3]{Q_1 Q_2 Q_3}, \ 
        e^{4 \pi \sqrt{-1}/3} \sqrt[3]{Q_1 Q_2 Q_3}.
\end{align*}
It is easy to see that all critical points are non-degenerate
and have the same valuation which lies in the interior of 
the Gelfand-Cetlin polytope.
Hence we have as many critical points
as $\dim H^*(\Fl(3)) = 6$
in this case.
One can show,
using the presentation of the quantum cohomology
in \cite[Theorem 1]{Givental-Kim},
that the set of eigenvalues
of the quantum cup product by $c_1(\Fl(3))$
coincides with the set of critical values
of the potential function.

The Floer differential $\frakm_1^b$ is trivial
for each critical point $(\bsu, \bsx)$ of $\po$,
and the corresponding Floer cohomology is given by
\[
 \HF((L(\bsu), b), (L(\bsu), b); \Lambda_0) \cong
  H^*(L(\bsu); \Lambda_0) \cong H^*(T^3; \Lambda_0).
\]
\end{example}

\begin{example}
We identify $\Gr(2, 4)$ with the adjoint orbit of
$\diag (2\lambda, 2 \lambda, 0, 0)$.
Setting $Q = T^{2\lambda}$, the potential function is given by
\begin{align}
%\begin{split}
 \po
  &= e^{- x_2} T^{- u_2 + 2\lambda}
     + e^{- x_1 + x_2} T^{- u_1 + u_2}
     + e^{x_1 - x_3} T^{u_1 - u_3}
     \nonumber \\
  & \qquad
     + e^{x_3} T^{u_3}
     + e^{x_2 - x_4} T^{u_2 - u_4}
     + e^{- x_3 + x_4} T^{- u_3 + u_4}
     \nonumber \\
  &= \frac{Q}{y_2} + \frac{y_2}{y_1} + \frac{y_1}{y_3}
     + y_3 + \frac{y_2}{y_4} + \frac{y_4}{y_3}.
%  &= Q_1 y_1^{-1} + Q_2^{-1} y_1 + Q_2 y_2^{-1}
%     + Q_3^{-1} y_2 + y_1 y_3^{-1} + y_2^{-1} y_3 \\
%  &= \frac{Q_1}{X} + \frac{X}{Q_2} + \frac{Q_2}{Y}
%     + \frac{Y}{Q_3} + \frac{X}{Z} + \frac{Z}{Y}
%\end{split}
 \label{eq:pot_24}
\end{align}
This function has four critical points
\[
 (y_1, y_2, y_3, y_4) = 
 \left( (-1)^i \sqrt[4]{Q^2}, \sqrt{-1}^i \sqrt[4]{\frac{Q^3}{4}},
       \sqrt{-1}^i \sqrt[4]{4Q}, (-1)^i \sqrt[4]{Q^2} \right) 
\]
for $i=0, 1, 2, 3$, and the corresponding critical values are
\begin{align} \label{eq:critv_24}
 \po = 4 \sqrt{2} \sqrt{-1}^i \sqrt[4]{Q}.
\end{align}
Since $\dim H^*(\Gr(2, 4)) = 6$,
one has less critical point than $\dim H^*(\Gr(2, 4))$.
These critical points are non-degenerate and 
have a common valuation
\[
  \bsu_0 = (\lambda, 3\lambda /2, \lambda/2, \lambda)
  \in \Int \Delta.
\]
Hence there exist four weak bounding cochains $b_0, \dots, b_3$
such that
\[
 \HF((L(\bsu_0), b_i), (L(\bsu_0), b_i); \Lambda_0) 
  \cong H^*(L(\bsu_0); \Lambda_0)
  \cong H^*(T^4; \Lambda_0)
\]
for $i=0,1,2,3$.
One can show,
using the presentation of the quantum cohomology
in \cite[Theorem 0.1]{MR1621570},
that the set eigenvalues of the quantum cup product by $c_1(\Gr(2,4))$
consists of the four critical values of the potential function
and the zero eigenvalue with multiplicity two.
\end{example}

\begin{example}
We identify $\Gr(2,5)$ with the adjoint orbit of
$\diag (\lambda, \lambda, 0,0,0)$.
Since the Gelfand-Cetlin polytope is defined by \eqref{eq:GCpattern_Gr(2,5)},
the potential function is given by
\begin{align} \label{eq:pot_25}
  \po = \frac{Q}{y_2} + \frac{y_2}{y_1} + \frac{y_1}{y_3}
      + \frac{y_2}{y_4} + \frac{y_4}{y_3} + \frac{y_3}{y_5}
      + y_5 + \frac{y_4}{y_6} + \frac{y_6}{y_5}.
\end{align}
This function has ten critical points
defined by
\[
 y_6^5 = Q^5, \quad
 Q y_4 = y_6(y_6^3-y_4^2),
\]
and
\[
 y_1 = \frac{Q}{y_6}, \quad
 y_2 = \frac{Q}{y_5}, \quad
 y_3 = \frac{Q}{y_4}, \quad
 y_5 = \frac{y_6^2}{y_4}.
\]
The set 
\begin{align} \label{eq:critv_25}
 \lc 5 (\zeta_5^i + \zeta_5^j) Q^{1/5} \relmid
 \zeta_5 = \exp(2 \pi \sqrt{-1}/5) \text{ and }
 0 \le i < j \le 4 \rc
\end{align}
of critical values of the potential function coincides
with the set of eigenvalues of the quantum cup product
by $c_1(\Gr(2,5))$.
\end{example}

\section{Floer cohomologies of non-torus fibers} \label{sc:HF_nontorus}

We briefly recall the construction of the $A_{\infty}$ structure 
$\{ \frakm_k \}_{k \ge 0}$, 
omitting various technical details.
Let $L$ be a spin, oriented, and compact Lagrangian submanifold 
 in a symplectic manifold $(X, \omega)$.
For an almost complex structure $J$ compatible with $\omega$, 
let $\scM_{k+1}(J, \beta)$ be the moduli space
of stable $J$-holomorphic maps
$v : (\Sigma, \partial \Sigma) \to (X, L)$ 
from a bordered Riemann surface $\Sigma$ 
in the class $\beta \in \pi_2(X, L)$ of genus zero 
with $(k+1)$ boundary marked points 
$z_0, z_1, \dots, z_k \in \partial \Sigma$.
Then 
$\frakm_k = \sum_{\beta \in \pi_2(X, L)} T^{\beta \cap \omega} \frakm_{k, \beta} 
\colon H^*(L; \Lambda_0)^{\otimes k} \to H^*(L; \Lambda_0)$ 
is defined by
\[
  \frakm_{k, \beta} (x_1, \dots, x_k) 
  = (\ev_0)_* (\ev_1^* x_1 \cup \dots \cup \ev_k^* x_k),
\]
where $\ev_i \colon \scM_{k+1}(J, \beta) \to L$,
$[v, (z_0, \dots, z_k)] \mapsto v(z_i)$ 
is the evaluation map at the $i$th marked point.

%We have seen in Corollary \ref{cr:dislpaceable_Gr(2,4)} 
%and Proposition \ref{pr:dislpaceable_Gr(2,5)}
%that $L_t \subset \Gr(n, 2n)$ for $t \ne 0$ and 
%$L_i(s_1, s_2,t) \subset \Gr(2,5)$ for $i=1, 2$ are displaceable.
%From the Hamiltonian invariance of Floer cohomologies, 
%we have the following.
%
%\begin{corollary}
%For $L=L_t$ in $\Gr(n,2n)$ with $t \ne 0$ or 
%$L= L_i(s_1, s_2, t)$ in $\Gr(2,5)$,  we have
%\[
% \HF((L,b), (L,b); \Lambda) = 0.
%\]
%\end{corollary}

\subsection{Holomorphic disks in $(\Fl(3), L_0)$}

We identify $\Fl(3)$
with the adjoint orbit of $\diag (\lambda_1, 0, -\lambda_2)$
for $\lambda_1, \lambda_2 >  0$ 
as in Subsection \ref{sc:fiber_Fl(3)}.
Note that the symplectic form and the first Chern class are given by 
$
 \omega = \lambda_1 \omega_{\bP_1} + \lambda_2 \omega_{\bP_2}
$
and
$c_1(\Fl(3)) = 2(\omega_{\bP_1} + \omega_{\bP_2})$,
respectively.

Recall that the homotopy group $\pi_2(\Fl(3)) \cong \bZ^2$ is generated by
1-dimensional Schubert varieties $X_1$ and $X_2$, 
which are rational curves of bidegree $(1,0)$ and $(0,1)$
in $\bP_1 \times \bP_2 \cong \bP^2 \times \bP^2$, respectively.
Since $L_0$ is diffeomorphic to $\SU(2) \cong S^3$, 
we have $\pi_1(L_0) = \pi_2(L_0) =0$.
The long exact sequence of homotopy groups yields
\[
  \pi_2(\Fl(3), L_0) \cong \pi_2(\Fl(3)) \cong \bZ^2.
\]
Let $\beta_1$, $\beta_2$ be generators of $\pi_2(\Fl(3), L_0)$
corresponding to $X_1$ and $X_2$, respectively.
The symplectic area of $\beta_i$ is given by
\[
  \beta_i \cap \omega 
  = [X_i] \cap (\lambda_1 \omega_{\bP_1} + \lambda_2 \omega_{\bP_2})
  = \lambda_i.
\]

Let $\tau$ be the anti-holomorphic involution on $\Fl(3)$ 
defined in \eqref{eq:involution_Fl(3)}.
For a holomorphic disk $v : (D^2, \partial D^2) \to (\Fl(3), L_0)$, 
we define a new holomorphic disk
$\tau_* v: (D^2, \partial D^2) \to (\Fl(3), L_0)$ by
\[
  \tau_* v(z) = \tau (v (\overline{z})).
\]
Since $L_0$ is the fixed point set of $\tau$, one can glue $v$ and $\tau_* v$
along the boundary to obtain a holomorphic curve
$w = v \# \tau_* v : \bP^1 \to \Fl(3)$.
The induced involution on $\pi_2(\Fl(3), L_0)$, which is also denoted by $\tau_*$, 
is given by $\tau_*\beta_1 = \beta_2$.
If $v$ represents $\beta_1$ or $\beta_2$, 
then $[w] = \beta_1 + \beta_2 = [X_1] + [X_2]$,
i.e., $w$ is a rational curve of bidegree $(1,1)$.

Let $\mu_{L_0} : \pi_2(\Fl(3), L_0) \to \bZ$ be the Maslov index.
If we assume $\lambda_1 = \lambda_2$ so that 
$\tau$ is an anti-symplectic involution,
then we have
\[
  \mu_{L_0}(\beta_i) 
  = \frac 12 (\mu_{L_0}(\beta_i) + \mu_{L_0}(\tau_* \beta_i))
  = ( [X_1]+ [X_2] ) \cap c_1(\Fl(3))  = 4
\]
for $i=1, 2$.
Since the symplectic form $\omega$
and the Lagrangian submanifold $L_0$ depend
continuously on $\lambda_1, \lambda_2  >0$, 
the Maslov index
$\mu_{L_0}(\beta_1) = \mu_{L_0}(\beta_2) = 4$ is independent of
$\lambda_1, \lambda_2$.

To describe holomorphic disks with Lagrangian boundary condition,
we identify the unit disk $D^2$ with the upper half plane 
$\bH = \bH_+$.

\begin{proposition} \label{pr:disk_Fl}
 Let $w \colon \bP^1 \to \Fl(3)$ be a holomorphic curve of bidegree $(1,1)$ 
 such that $w(\bR \cup \{\infty \}) \subset L_0$.
 After the $\SU(2)$-action, we may assume
 \begin{equation}
   w(\infty)= ([1:0:\sqrt{\lambda_1/\lambda_2}], 
    [1:0:- \sqrt{\lambda_2/\lambda_1}]).
   \label{eq:bdry_cond_Fl(3)}
 \end{equation}
 We can write
 \begin{equation}
   w(0) = \Bigl( \bigl[a_1:a_2: \sqrt{\lambda_1/\lambda_2} \bigr], 
     \bigl[\overline{a}_1: \overline{a}_2: 
         - \sqrt{\lambda_2/\lambda_1} \bigr]
     \Bigr)
   \in L_0
   \label{eq:bdry_cond_Fl(3)_2}
 \end{equation}
 for some $(a_1, a_2) \in S^3 \setminus \{ (1,0) \}$.
 Then $w$ is given by
 \[
   w(z) = \Bigl( 
     \bigl[cz+a_1: a_2: \sqrt{\lambda_1/\lambda_2} (cz+1) \bigr], 
     \bigl[\overline{c}z+\overline{a}_1: \overline{a}_2: 
   - \sqrt{\lambda_2/\lambda_1} (\overline{c}z+1) \bigr] \Bigr)
 \]
 with $c/ \overline{c} = - (a_1-1)/(\overline{a}_1-1)$.
\end{proposition}

\begin{remark}
  After the action of 
  \[
%    \Aut (\bH; 0, \infty) := 
    \{ g \in \PSL(2, \bR) \,|\, g(0)=0, \, g(\infty) = \infty \}
    \cong \bR_{>0}
  \]
  on $\bH$, we may assume that $|c|=1$.
\end{remark}

\begin{proof}
The assumptions \eqref{eq:bdry_cond_Fl(3)} and \eqref{eq:bdry_cond_Fl(3)_2} 
implies that $w$ has the form
 \[
   w(z) = \Bigl( 
     \bigl[c_1z+a_1: a_2: \sqrt{\lambda_1/\lambda_2} (c_1z+1) \bigr], 
     \bigl[c_2 z+ \overline{a}_1: \overline{a}_2: 
   - \sqrt{\lambda_2/\lambda_1} (c_2 z+1) \bigr] \Bigr)
 \]
for some $c_1, c_2 \in \bC^*$.
The Pl\"ucker relation 
\begin{align*}
  0 &= - (c_1 z + a_1)(c_2 z + \overline{a}_1) - |a_2|^2
       + (c_1 z + 1)(c_2 z + 1) \\
    &= (c_1 - a_1 c_1 + c_2 - \overline{a}_1 c_2) z
\end{align*}
implies $c_1(\overline{a}_1 -1) + c_2 (a-1) =0$.
On the other hand, the Lagrangian boundary condition 
$w(\bR) \subset L_0$ implies that
\[
  \frac{c_1x+a_1}{c_1x+1} = \frac{\overline{c_2}x+a_1}{\overline{c_2}x+1},
  \quad
  \frac{a_2}{c_1x+1} = \frac{a_2}{\overline{c_2}x+1},
  \quad x \in \bR,
\]
which means $c_2 = \overline{c_1}$.
\end{proof}

Note that  $\arg c$ is determined by $a_1$ up to sign, 
and the sign corresponds to whether 
$v = w|_{\bH}$ represents $\beta_1$ or $\beta_2$.
Namely any holomorphic disk in the class $\beta_i$ satisfying 
\eqref{eq:bdry_cond_Fl(3)} and \eqref{eq:bdry_cond_Fl(3)_2} 
is uniquely determined by $(a_1, a_2)$ for $i=1, 2$.

\begin{example}
  Suppose that $(a_1, a_2) = (-1, 0)$.
  Then $c = \pm \ii$, and the corresponding holomorphic disks 
  are given by
  \[
     v_{\pm}(z) = \Bigl( 
       \bigl[ z \pm \ii : 0 : 
         \sqrt{\frac{\lambda_1}{\lambda_2}} (z \mp \ii) \bigr], 
       \bigl[z \mp \ii : 0 : 
       - \sqrt{\frac{\lambda_2}{\lambda_1}} (z \pm \ii ) \bigr] \Bigr).
  \]
  It is easy to see that the image $v_+(\bH)$ (resp. $v_-(\bH)$) 
  is the inverse image of the edge of $\Delta$ given by
  $u^{(1)}_1 = u^{(2)}_1$ and $u^{(2)}_2 = 0$
  (resp. $u^{(1)}_1 = u^{(2)}_2$ and $u^{(2)}_1 = 0$),
  which is the upper (resp. lower) vertical edge emanating from the vertex
  $\bszero = (0,0,0)$.
  The generators $\beta_1, \beta_2$ of $\pi_2(\Fl(3), L_0)$ 
  are represented by $v_+$ and $v_-$ respectively.
\end{example}

%%%%%%%%%%%%%%%%%%%%%%%%%%%%%%%%%%%%%%%%%%%%%%%%%%%%%%%%%%%%%%%%%%%

\subsection{Floer cohomology of the $\SU(2)$-fiber in $\Fl(3)$}

Let $J$ be the standard complex structure on $\Fl(3)$.
Since the fiber $L_0$ is $\SU(2)$-homogeneous, 
\cite[Proposition 3.2.1]{1401.4073} implies the following.

\begin{proposition}
Any $J$-holomorphic disk in $(\Fl(3), L_0)$ is Fredholm regular.
Hence the moduli space $\scM_{k+1}^{\reg}(J, \beta)$ 
of $J$-holomorphic disks in the class $\beta$ 
with $k+1$ boundary marked points
is a smooth manifold of dimension
\begin{align*}
  \dim \scM_{k+1}^{\reg}(J, \beta) 
  &= \dim L_0 + \mu_{L_0}(\beta) + k+1 -3 \\
  &= \mu_{L_0}(\beta) + k+1.
\end{align*}
\end{proposition}

In particular, we have $\dim \scM_2(J, \beta_i) = 6$ for $i=1,2$.
Proposition \ref{pr:disk_Fl} implies the following:

\begin{corollary} \label{cr:ev_Fl}
  Let $U = \SU(2) \setminus \{ 1 \} \cong \{ (a_1, a_2) \in S^3 \, | \, a_1 \ne 1 \}$.
  Then $\scM_2(J, \beta_i)$ has an open dense subset
  diffeomorphic to $\SU(2) \times U$ on which the evaluation map
  is given by
  \[
    \SU(2) \times U \longrightarrow 
    L_0 \times L_0 \cong \SU(2) \times \SU(2),
    \quad
    (g_1, g_2) \longmapsto (g_1, g_1 g_2).
  \]
  In particular, $\ev : \scM_2(J, \beta_i) \to L_0 \times L_0$
  is generically one-to-one.
\end{corollary}

Since the minimal Maslov number is 
$\mu_{L_0}(\beta_1) = \mu_{L_0}(\beta_2) = 4$ and
\[
  \deg \frakm_{1, \beta}(x) = \deg x + 1 -\mu_{L_0}(\beta),
  \quad x \in H^*(L_0; \Lambda_0),
\]
the only nontrivial parts of the Floer differential are 
\[
  \frakm_{1, \beta_i}: 
  H^3 (L_0) \cong H_0 (L_0) \longrightarrow
  H^0 (L_0) \cong H_3 (L_0)
\]
for $i=1,2$.
Corollary \ref{cr:ev_Fl} implies that
for the class $[p] \in H_0(L_0)$ of a point, we have
\[
  \frakm_{1, \beta_i}([p]) = 
  {\ev_0}_* [\scM_2(J, \beta_i) {}_{\ev_1}\!\times \{p\} ]
  = \pm [L_0].
\]
To see the sign, we use a result on the orientation of the moduli spaces 
of pseudo-holomorphic disks by 
Fukaya, Oh, Ohta, and Ono \cite[Theorem 1.5]{0912.2646}.
The following statement is a slightly weaker version of the result, 
which is sufficient for our purpose.

\begin{theorem} \label{th:FOOO_orientation}
Let $(X, \omega)$ be a compact symplectic manifold,
and $\tau$ an anti-symplectic involution on $X$ whose fixed point set 
$L= \Fix(\tau)$ is non-empty, compact, connected, and spin.
Then $\frakm_{k,\beta}$ and $\frakm_{k, \tau_* \beta}$ satisfy
\[
  \frakm_{k,\beta} (P_1, \dots, P_k) 
  = (-1)^{\epsilon} \frakm_{k, \tau_* \beta}(P_k, \dots, P_1),
\]
where 
\[
  \epsilon = \frac{\mu_L(\beta)}{2} + k+1 + 
             \sum_{1 \le i< j \le k} (\deg P_i-1)(\deg P_j-1).
\]
\end{theorem}

\begin{corollary} \label{cr:orientation_Fl(3)}
We have $\frakm_{1,\beta_1} = \frakm_{1,\beta_2}$ 
for general $\lambda_1, \lambda_2 >0$.
\end{corollary}

\begin{proof}
If $\lambda_1 = \lambda_2$, then $\tau$ is anti-symplectic,
and thus Theorem \ref{th:FOOO_orientation} implies
\begin{equation}
  \frakm_{1, \beta_1} 
  = (-1)^{\mu_{L_0}(\beta_1)/2 + 2} \frakm_{1, \tau_* \beta_1}
  = \frakm_{1, \beta_2}.
  \label{eq:sign_m1}
\end{equation}
Corollary \ref{cr:ev_Fl} implies that 
$\scM_2(J, \beta_i)$ depends continuously on $\lambda_1, \lambda_2$, 
and hence its orientation is independent of $\lambda_1, \lambda_2$.
Thus \eqref{eq:sign_m1} holds for general $\lambda_1, \lambda_2$.
\end{proof}

Then we have
\[
  \frakm_1 ([p]) = 
  \sum_{i=1}^2 \frakm_{1, \beta_i}([p]) T^{\beta_i \cap \omega}
  = \pm (T^{\lambda_1} + T^{\lambda_2}) [L_0],
\]
which implies the following.

\begin{theorem} \label{th:Fl(3)}
The Floer cohomology of $L_0$ over the Novikov ring $\Lambda_0$
is 
\[
  \HF(L_0, L_0; \Lambda_0) \cong 
  \Lambda_0/ T^{\min \{ \lambda_1, \lambda_2 \}} \Lambda_0.
\]
%Hence the Floer cohomology over the Novikov field $\Lambda$ 
%is trivial:
%\[
% \HF(L_0, L_0; \Lambda) = 0.
%\]
\end{theorem}

\pref{th:main1} is an immediate consequence
of \pref{th:Fl(3)}.

%%%%%%%%%%%%%%%%%%%%%%%%%%%%%%%%%%%%%%%%%%%%%%%%%%%%%%%%%%%%%%%%%%%%%%

\subsection{Holomorphic disks in $(\Gr(2,4), L_t)$}

We identify $\Gr(2,4)$ with the adjoint orbit of 
$\diag (\lambda, \lambda, -\lambda, -\lambda)$ for $\lambda >0$.
Note that the Kostant-Kirillov form and
the first Chern class are given by
\[
  \omega = 2 \lambda \omega_{\FS},
  \quad
  c_1(\Gr(2,4)) = 4 \omega_{\FS}, 
\]
respectively,
where $\omega_{\FS}$ is the Fubini-Study form 
on $\bP (\bigwedge^2 \bC^4)$.

Recall that $\pi_2 (\Gr(2,4)) \cong \bZ$ is generated by
a 1-dimensional Schubert variety $X_1$, which is a rational curve 
of degree one in $\bP (\bigwedge^2 \bC^4)$.
Since $\pi_1( \Gr(2,4)) = \pi_2 (L_t) = 0$ and
$\pi_1(L_t) \cong \bZ$, the exact sequence
\[
  0 \longrightarrow \pi_2 (\Gr(2,4)) \longrightarrow
  \pi_2(\Gr(2,4), L_t) \longrightarrow \pi_1(L_t) 
  \longrightarrow 0
\]
implies that $\pi_2(\Gr(2,4), L_t) \cong \bZ^2$.
Let $\beta_1, \beta_2$ be generators of $\pi_2(\Gr(2,4), L_t)$
such that $\beta_1 + \beta_2 = [X_1] \in \pi_2(\Gr(2,4))$.

\begin{example}
Consider a holomorphic curve $w : \bP^1 \to \Gr(2,4)$ of degree one 
defined by
\begin{equation}
  w(z) =  \left[ \sqrt{\frac{\lambda +t}{\lambda -t}}(z- \ii)
      : 0 : z - \ii : - z- \ii : 0
      : \sqrt{\frac{\lambda -t}{\lambda +t}}(z + \ii) \right].
      \label{eq:Schubert_Gr(2,4)}
\end{equation}
Since $w$ maps $\bR \cup \{\infty \}$ to $L_t$, 
the restrictions 
\begin{align*}
  v_+ &= w|_{\bH_+} : (\bH_+, \partial \bH_+) \longrightarrow (\Gr(2,4), L_t), \\
  v_- &= w|_{\bH_-} : (\bH_-, \partial \bH_-) \longrightarrow (\Gr(2,4), L_t)
\end{align*}
to the upper and lower half planes give holomorphic disks 
representing $\beta_1$ and $\beta_2$.
We define $\beta_1 = [v_+]$ and $\beta_2 = [v_-]$.
It is easy to see that the symplectic areas of $v_{\pm}$ are given by
\[
  \omega (\beta_1) = \int_{\bH_+} v_+^* \omega = \lambda + t,
  \quad
  \omega (\beta_2) = \int_{\bH_-} v_-^* \omega = \lambda - t.
\]
In the case where $t=0$, the disk $v_+$ sends $\ii \in \bH$ to
$v_+(\ii) = [0: 0: 0: -1: 0: 1]$,
which is in the fiber $\Phi^{-1}(\bsu_1)$ over the point 
$\bsu_1 \in \Delta$ defined by
$u_1^{(2)} = u_1^{(1)} = \lambda$ and
$u_2^{(2)} = 0$ (see Figure \ref{fg:polytope(2,4)}).
On the other hand, $v_-(-\ii) =[1: 0: 1: 0: 0: 0]$ lies on the fiber 
over the point $\bsu_2 \in \Delta$ defined by 
$u_2^{(2)} = u_1^{(1)} = -\lambda$ and
$u_1^{(2)} = 0$.
\end{example}

Let $\tau_t$ be the anti-holomorphic involution on $\Gr(2,4)$ 
defined in \eqref{eq:involution_Gr(2,4)}.
Note that $(\tau_t)_*$ is given by 
$(\tau_t)_*v (z) = \tau_t (v( - \overline{z}))$
for $v: (\bH, \partial \bH) \to (\Gr(2,4), L_t)$. 
Since $(\tau_t)_* v_+ = v_-$, the induced involution
on $\pi_2(\Gr(2,4), L_t)$ is given by
$(\tau_t)_* \beta_1 = \beta_2$.
Then the Maslov index of $\beta_i$ is given by
\[
  \mu_{L_t}(\beta_i) 
  = \frac 12 \left( \mu_{L_t}(\beta_i) + \mu_{L_t}((\tau_t)_* \beta_i) \right) 
  =  [X_1] \cap c_1(\Gr(2,4))   
  = 4
\] 
for $i=1,2$.

We describe holomorphic curves $w \colon \bP^1 \to \Gr(n, 2n)$ of 
degree one such that $w (\bR \cup \{\infty\})$ is contained 
in the Lagrangian fiber $L_t$.
\pref{pr:rat-curve_Gr} below is taken from
\cite[Theorem 2.1]{MR1837108},
which is well-known in control theory
(cf.~e.g.~\cite{MR0325201}).

\begin{proposition} \label{pr:rat-curve_Gr}
Suppose that a holomorphic curve $w \colon \bP^1 \to \Gr(k,n) = \Vtilde (k,n) / \GL(k, \bC)$ 
of degree $d$ is given by
\[
  w \colon z \longmapsto 
  \begin{pmatrix} I_k \\ F(z) \end{pmatrix}
  \mod \GL(k, \bC)
\]
for a rational function $F(z)$ with values in $(n-k) \times K$ matrices.
Then there exist matrix valued polynomials $P(z)$, $Q(z)$ 
of size $k \times k$ and $(n-k) \times k$ respectively
such that 
\begin{enumerate}
  \item $F(z) = Q(z) P(z)^{-1}$, i.e., the curve $w$ is given by
        \[
          w \colon z \longmapsto 
          \begin{pmatrix} P(z) \\ Q(z) \end{pmatrix}
          \mod \GL(k, \bC),
        \]
  \item $P(z)$ and $Q(z)$ are coprime in the sense there exist matrix valued
        polynomials $X(z)$, $Y(z)$ such that 
        $X(z) P(z) + Y(z) Q(z) = I_k$, and
  \item $\deg (\det P(z)) = d$.
\end{enumerate}
Such $P(z)$ and $Q(z)$ are unique up to multiplication of elements 
in $\GL(k, \bC [z])$.
\end{proposition}

Note that \eqref{eq:U(n)fiber_in_V} implies that
the $\U(n)$-fiber $L_t \subset \Gr(n, 2n) = \Vtilde (n, 2n)/ \GL(n, \bC)$ consists of 
\[
  \begin{pmatrix} I_n \\
  \sqrt{(\lambda - t)/(\lambda + t)} \, A
  \end{pmatrix}
  \mod \GL(n, \bC)
\]
for $A \in \U(n)$.

\begin{proposition}\label{pr:disk_Gr}
  Let $w \colon \bP^1 \to \Gr(n,2n)$ be a holomorphic curve of degree one
  such that $w (\bR \cup \{ \infty \}) \subset L_t$,
  and let $F(z)$ denote the corresponding rational function with values
  in $n \times n$ matrices.
  By the $\U(n)$-action, we assume that 
  \begin{equation}
    F(\infty) = \sqrt{\frac{\lambda - t}{\lambda + t}} I_n
    \in \sqrt{\frac{\lambda - t}{\lambda + t}} \U(n),
  \label{eq:bdry_cond_at_infty}
  \end{equation}
  and set
  \begin{equation}
    F(0) = \sqrt{\frac{\lambda - t}{\lambda + t}} A
    \label{eq:bdry_cond_at_zero}
  \end{equation}
   for $A \in \U(n)$.
  Then there exist 
  \[
    a = \begin{pmatrix} a_1 \\ \vdots \\ a_n \end{pmatrix}
    \in S^{2n-1}/S^1 = \bP^{n-1}
  \]
  and $c \in \bC \setminus \bR$ such that 
  \[
    A = I_n + \left( \frac{c^2}{|c|^2} -1 \right) a a^*,
  \]
  and 
  \begin{equation}
    F(z) = \sqrt{\frac{\lambda - t}{\lambda + t}}
    \frac 1{z - \overline{c}} (z I_n - \overline{c} A)
    = \sqrt{\frac{\lambda - t}{\lambda + t}}
    \left( I_n - \frac{c - \overline{c}}{z - \overline{c}} aa^* \right) .
  \label{eq:disk_Gr(n,2n)}
  \end{equation}
\end{proposition}

\begin{proof}
Let $F(z) = Q(z)P(z)^{-1}$ be the factorization given in \pref{pr:rat-curve_Gr}.
Then 
the assumptions \eqref{eq:bdry_cond_at_infty}, \eqref{eq:bdry_cond_at_zero},
and $\deg (\det P(z)) =1$ imply that
$F(z)$ has the form
\[
    F(z) = \sqrt{\frac{\lambda - t}{\lambda + t}}
    \frac 1{z - \overline{c}} (z I_n - \overline{c} A)
\]
for some $c \in \bC$.
The Lagrangian boundary condition $w(\bR \cup \{ \infty \}) \subset L_t$
implies that
\[
  \frac 1{x - \overline{c}} (x I_n - \overline{c} A) \in \U(n)
\]
for any $x \in \bR$, which means
$\overline{c} A + c A^* = (c + \overline{c}) I_n$,
or equivalently, $\overline{c}A - \re(c)I_n$ is skew-hermitian.
Hence $\overline{c}A - \re(c)I_n$ has pure imaginary eigenvalues 
$\sqrt{-1} \alpha_1, \dots, \sqrt{-1} \alpha_n$, and
can be diagonalized by some $g \in \U(n)$;
\[
  g^* (\overline{c}A - \re(c)I_n) g =
  \diag(\sqrt{-1}\alpha_1, \dots, \sqrt{-1}\alpha_n).
\]
Since 
\[
  g^* A g = \diag \left(
    \frac{\re (c) + \sqrt{-1}\alpha_1}{\overline{c}}, \dots, 
    \frac{\re (c) + \sqrt{-1}\alpha_n}{\overline{c}}
    \right) \in \U(n)
\]
has eigenvalues of unit norm,
we have $\alpha_i = \pm \im (c)$ for $i = 1, \dots, n$.
After the action of a permutation matrix, 
we may assume that $g^*Ag$ has the form
\begin{equation}
  g^* A g = \diag( \underbrace{c/\overline{c}, \dots, c/\overline{c}}_k,
            \underbrace{1, \dots, 1}_{n-k})
  =: C
    \label{eq:det_A}
\end{equation}
for some $k$.
Then $F(z)$ is given by
\[
  F(z) = \sqrt{\frac{\lambda - t}{\lambda + t}}
    \frac 1{z - \overline{c}}  g (z I_n - \overline{c} C) g^*
  = \sqrt{\frac{\lambda - t}{\lambda + t}}
  g \diag \left(
    \frac{z - c}{z - \overline{c}}, \dots, \frac{z - c}{z - \overline{c}},
    1, \dots, 1
    \right) g^*
\]
In particular, we have
\begin{equation*}
  \det F(z) = \left( \frac{\lambda - t}{\lambda + t} \right)^{n/2}
  \left( \frac{z - c}{z - \overline{c}} \right)^k.
\end{equation*}
The condition $\deg (\det P(z)) = 1$ implies that $k=1$, i.e.,
\[
  C = \diag (c/\overline{c}, 1, \dots, 1)
  = (c/\overline{c} - 1) E_{11} + I_n,
\]
where  $E_{11} = \diag (1, 0, \dots , 0) \in \mathfrak{gl}(n, \bC)$.
Let $a \in S^{2n-1} \subset \bC^n$ be the first column of $g$.
Then we have
\[
  A = g \left( \left( \frac{c^2}{|c|^2} - 1 \right) E_{11} + I_n \right) g^* 
  = \left( \frac{c^2}{|c|^2} -1 \right) a a^* + I_n,
\]
which proves the proposition.
\end{proof}

\begin{remark}
\begin{enumerate}
  \item The equation \eqref{eq:det_A} (with $k=1$) implies that 
          $\det A = c / \overline{c} = c^2/|c|^2$.
  \item After the $\bR_{>0}$-action on the domain, we may assume that $|c| = 1$.
\end{enumerate}
\end{remark}

We now assume that $n=2$.
The sign of $\im (c) = \im \sqrt{\det A}$ corresponds to the homotopy class
of the holomorphic disk $v = w|_{\bH}$.
The curve $w$ corresponding to $a = [1:0]$ and $c = -\sqrt{-1}$
coincides with \eqref{eq:Schubert_Gr(2,4)}, and hence
$w|_{\bH} = v_+$ represents $\beta_1$. 
Thus $v = w|_{\bH}$ represents $\beta_1$ (resp. $\beta_2$)
when $\im (c) = \im \sqrt{\det A} <0$ (resp. $\im (c) >0$).

%%%%%%%%%%%%%%%%%%%%%%%%%%%%%%%%%%%%%%%%%%%%%%%%%%%%%%%%%%%

\subsection{Floer cohomologies of the $\U(2)$-fibers in $\Gr(2,4)$}

Since the minimal Maslov number of the $U(2)$-fiber $L_t$ is 
$\mu_{L_t}(\beta_i) =4$, 
we have the following by degree reason.

\begin{lemma}
The potential function $\po \colon H^1(L_t; \Lambda_0) \to \Lambda_0$ 
for $L_t$ is trivial:
\[
  \po \equiv 0.
\]
\end{lemma}

The cohomology of $L_t \cong S^1 \times S^3$ is given by 
\[
  H^*(L_t) \cong H^*(S^1) \otimes H^*(S^3).
\]
Let $\bfe_1 \in H^1(L_t; \bZ) \cong H^1(S^1; \bZ)$ and
$\bfe_3 \in H^3(L_t; \bZ) \cong H^3(S^3; \bZ)$ be the generators,
and write $b = x \bfe_1 \in H^1(L_t; \Lambda_0)$.
Since $\deg \frakm_{1, \beta}^b = 1 - \mu_{L_t}(\beta)$ and
the minimal Maslov number is four,
the only nontrivial parts of the Floer differential $\frakm_1^b$ are
\begin{align*}
  \frakm_{1, \beta_i}^b &: H^4(L_t) \cong H^1(S^1) \otimes H^3(S^3)
     \longrightarrow H^1(L_t) \cong H^1(S^1), \\
  \frakm_{1, \beta_i}^b &: H^3(L_t) \cong  H^3(S^3)
     \longrightarrow H^0(L_t) \cong \Lambda_0
\end{align*}
for $i =1, 2$.

Since $(\Gr(2,4), L_t)$ is $\U(2)$-homogeneous,
any $J$-holomorphic disk is Fredholm regular 
for the standard complex structure $J$
by \cite[Proposition 3.2.1]{1401.4073}.
Hence one has $\dim \scM_2(J, \beta_i) = 7$ for $i=1,2$.
Now \pref{pr:disk_Gr} implies the following:

\begin{corollary} \label{cr:moduli_G(2,4)}
  Define $f \colon (0, 2\pi) \times \bP^1 \to \U(2)$ by
  $f(\theta, a) = (e^{\sqrt{-1}\theta} -1)aa^* + I_2$.
  For $i=1, 2$, 
  the moduli space $\scM_2(J, \beta_i)$ has an open dense subset
  diffeomorphic to $\U(2) \times (0, 2\pi) \times \bP^1$
  such that the evaluation map
  is given by
  \[
    \U(2) \times (0, 2\pi) \times \bP^1 \longrightarrow 
    L_t \times L_t \cong \U(2) \times \U(2),
    \quad
    (g, \theta, a) \longmapsto (g, g \cdot f(\theta, a)).
  \]
\end{corollary}

Note that $e^{\sqrt{-1}\theta} = \det f(\theta, a)$ is related to $c \in S^1$ 
in \pref{pr:disk_Gr} by 
$c = \exp (\sqrt{-1}(\theta /2 + \pi))$ or $c = \exp (\sqrt{-1}\theta /2)$
corresponding to $i=1, 2$.

Next we consider $\scM_{k+l+2}(J, \beta_i)$.
For a rational curve $w \colon \bP^1 \to \Gr(2,4)$ given by \eqref{eq:disk_Gr(n,2n)}, 
the composition 
$\det \circ w|_{\partial \bH} \colon \partial \bH = \bR \to L_t \cong \U(2) \to S^1$ 
is given by
\[
  x \longmapsto \frac{x - c}{x - \overline{c}}.
\] 
Hence each boundary point $x \in \partial \bH$ is determined by the argument of 
$\det w(x) = (x- c)/(x-\overline{c})$.
Fixing the $0$-th and $(k+1)$-st boundary marked points, 
we have the following.

\begin{corollary} \label{cr:moduli_G(2,4)_2}
  The moduli space $\scM_{k+l+2}(J, \beta_i)$ has an open dense subset
  diffeomorphic to 
  \[
    \biggl\{ (g, \theta, a, (t_i), (s_j)) 
      \in \U(2) \times (0, 2\pi) \times \bP^1 \times \bR^k \times \bR^l \biggm|  
      \begin{array}{c} 
      0< t_1 < \dots < t_k < \theta, \\
      \theta < s_1 < \dots < s_l < 2\pi
      \end{array}
      \biggr\}
  \]
  on which the evaluation maps 
  $\ev \colon  \scM_{k+l+2}(J, \beta_i) \to L_t \cong \U(2)$ satisfy
  \[
    (\ev_0, \ev_{k+1}) \colon (g, \theta, a, (t_i), (s_j)) 
    \longmapsto (g, g \cdot f(\theta, a))
  \]
  and
  \[
    \det \ev_i (g, \theta, a, (t_i), (s_j)) = 
    \begin{cases}
      e^{\sqrt{-1} t_i} \det g , & i=1, \dots, k,\\
      e^{\sqrt{-1} \theta} \det g , & i = k+1, \\
      e^{\sqrt{-1} s_{i-k-1}} \det g , & i=k+2, \dots, k+l+2.
    \end{cases}
  \]
\end{corollary}

\begin{theorem} \label{th:Gr(2,4)}
  For $b = x \bfe_1 \in H^1(L_0; \Lambda_0/2\pi \sqrt{-1} \bZ) 
  \cong \Lambda_0/2 \pi \sqrt{-1} \bZ$, 
  the deformed Floer differential $\frakm_1^b$ is given by
  \begin{align}
    \frakm_1^b (\bfe_3) 
    &= e^x T^{\lambda + t} + e^{-x} T^{\lambda - t}, 
        \label{eq:Fl_diff_1} \\
    \frakm_1^b (\bfe_1 \otimes \bfe_3) 
    &= (e^x T^{\lambda + t} + e^{-x} T^{\lambda - t}) \bfe_1.
      \label{eq:Fl_diff_2}
  \end{align}
  Hence the Floer cohomology of $(L_t, b)$ is 
  \[
 \HF((L_t, b), (L_t,b); \Lambda_0) \cong 
    \begin{cases}
      H^*(L_0; \Lambda_0) \quad 
        &\text{if $t = 0$ and $x= \pm \pi \sqrt{-1} /2$},\\
      (\Lambda_0 / T^{\min \{ \lambda-t, \lambda +t \}} \Lambda_0)^2 
        &\text{otherwise}.
    \end{cases}
  \]
  The Floer cohomology over the Novikov field is given by
  \[
 \HF((L_t, b), (L_t,b); \Lambda) \cong 
    \begin{cases}
      H^*(L_0; \Lambda) \quad 
        &\text{if $t = 0$ and $x= \pm \pi \sqrt{-1} /2$},\\
      0 &\text{otherwise}.
    \end{cases}
  \]
\end{theorem}

Recall that $\bfe_1, \bfe_3 \in H^*(\U(2))$ are given by
\[
  \bfe_1 = \frac{1}{2\pi \sqrt{-1}} \tr (g^{-1} d g)
           = \frac{1}{2\pi \sqrt{-1}} d \log (\det g),
  \quad
  \bfe_3 = \frac{1}{24 \pi^2} \tr \ld (g^{-1} dg)^3 \rd,
\]
where $g^{-1}dg$ is the left-invariant Maurer-Cartan form on $\U(2)$.

\begin{lemma}
For 
$f(\theta, a)  = (e^{\sqrt{-1}\theta} -1)aa^* + I_2$, we have
\begin{align}
  f^*\bfe_1 &= \frac{1}{2\pi} \tr (f^{-1} df) 
    =  \frac{d \theta}{2\pi} ,
  \label{eq:A^*e_1} \\
  f^*\bfe_3 &= \frac{1}{24\pi^2} \tr (f^{-1} df)^3
    = (1 - \cos \theta) \frac{d \theta}{2\pi} \wedge \omega_{\bP^1},
  \label{eq:A^*e_3}
\end{align}
where $\omega_{\bP^1}$ is the Fubini-Study form on $\bP^1$
normalized in such a way that
\[
  \int_{\bP^1} \omega_{\bP^1} = 1.
\]
\end{lemma}

\begin{proof}
The first assertion \eqref{eq:A^*e_1} follows from 
$\det f = e^{\sqrt{-1} \theta}$.
Since $f$ is $\SU(2)$-equivariant with respect to 
the natural action on $\bP^1$ and 
the adjoint action on $\U(2)$,
it suffices to show \eqref{eq:A^*e_3} at $a = [1: 0] \in \bP^1$.
A direct calculation gives
\[
  f^{-1}df = \begin{pmatrix}
     \sqrt{-1}d \theta & - (e^{-\sqrt{-1} \theta} -1) d \overline{a}_2 \\
    (e^{\sqrt{-1} \theta} -1) da_2 & 0
  \end{pmatrix},
\]
so that
\[
  \tr (f^{-1}df)^3  = 3 (2 - e^{\sqrt{-1} \theta} -  e^{- \sqrt{-1} \theta})
     \sqrt{-1}d \theta \wedge da_2 \wedge d \overline{a}_2
\]
at $a = [1:0]$.
On the other hand, the Fubini-Study form on $\bP^1$ is given by
\[
  \omega_{\bP^1} = \frac{\sqrt{-1}}{2\pi} da_2 \wedge d \overline{a}_2
\]
at $a = [1 : 0]$, which proves \eqref{eq:A^*e_3}.
\end{proof}

\begin{proof}[Proof of \pref{th:Gr(2,4)}]
Note that for $m : \U(2) \times \U(2) \to \U(2)$, $(g_1, g_2) \mapsto g_1 g_2$,
we have $m^*\bfe_i = \pi_1^* \bfe_i + \pi_2^* \bfe_i$ for $i=1, 3$, 
where $\pi_1, \pi_2 \colon \U(2) \times \U(2) \to \U(2)$ are the projections 
to the first and the second factors.
Then $\ev_j^* \bfe_i$ are given by
\begin{align*}
  \ev_i^* \bfe_1 &= \frac{1}{2\pi} dt_i + g^* \bfe_1, 
  \quad i = 1, \dots, k,\\
  \ev_{k+1+i}^* \bfe_1 &= \frac{1}{2\pi} dt_i + g^* \bfe_1, 
  \quad i = 1, \dots, l, \\
  \ev_{k+1}^* \bfe_3 &= f^* \bfe_3 + g^* \bfe_3
    = (1 - \cos \theta) \frac{d \theta}{2\pi} \wedge \omega_{\bP^1} + g^* \bfe_3,
\end{align*}
where $g^* \bfe_i$ is the pull-back of $\bfe_i$ by the projection 
\[
  \U(2) \times (0, 2\pi) \times \bP^1 \longrightarrow \U(2),
  \quad (g, \theta, a) \longmapsto g
\] 
to the first factor.
For $\theta \in (0, 2\pi)$, set 
\begin{align*}
  D_1(\theta) &= \{ (t_1, \dots, t_k) \in \bR^k  \mid
    0< t_1 < \dots < t_k < \theta \}, \\
  D_2(\theta) &= \{ (s_1, \dots, s_l) \in \bR^l  \mid   
    \theta < s_1 < \dots < s_l < 2\pi \}.
\end{align*}
Taking a suitable orientation on $\scM_{k+l+2}(\beta_1, J)$,
we have from Corollary \ref{cr:moduli_G(2,4)_2} that 
\begin{align}
  &\frakm_{k+l+1, \beta_1} ( \underbrace{b, \dots, b}_k, \bfe_3, 
    \underbrace{b, \dots, b}_l ) \notag \\
  & \quad = \int_{(0, 2\pi) \times \bP^1}
     \left( \int_{D_1(\theta)} \left(\frac{x}{2\pi} \right)^k dt_1 
          \wedge \dots \wedge dt_k \right) 
     \left( \int_{D_2(\theta)} \left(\frac{x}{2\pi} \right)^l ds_1 
          \wedge \dots \wedge ds_l \right)
      (1 - \cos \theta) \frac{d \theta}{2\pi} \wedge \omega_{\bP^1}  \notag \\
  & \quad = \int_{(0, 2\pi)}
     \frac{1}{k!} \left( \frac{\theta}{2\pi} \cdot x \right)^k 
     \frac{1}{l!} 
     \left( \left(1-\frac{\theta}{2\pi} \right) x\right)^l
      (1 - \cos \theta) \frac{d \theta}{2\pi}.
  \label{eq:int_along_fiber}
\end{align}
Hence 
\begin{align*}
  \frakm_{1, \beta_1}^b (\bfe_3) 
  &= \int_0^{2\pi} \sum_{k, l \ge 0}
     \frac{1}{k!} \left( \frac{\theta}{2\pi} \cdot x \right)^k 
     \frac{1}{l!} 
     \left( \left(1-\frac{\theta}{2\pi} \right) x\right)^l 
     (1 - \cos \theta) \frac{d \theta}{2\pi} \\
  &= \int_0^{2\pi} e^{(\theta / 2\pi)x}
                e^{(1- \theta / 2\pi)x} (1 - \cos \theta) \frac{d \theta}{2\pi} \\
  &= \int_0^{2\pi} e^x  (1 - \cos \theta) \frac{d \theta}{2\pi} \\
  &= e^x.
\end{align*}
The same argument as the proof of \pref{cr:orientation_Fl(3)} gives
\begin{align*}
  \frakm_{k+l+1, \beta_2} ( \underbrace{b, \dots, b}_k, \bfe_3, 
    \underbrace{b, \dots, b}_l ) 
  &= (-1)^{k+l}
    \frakm_{k+l+1, \beta_1} ( \underbrace{b, \dots, b}_l, \bfe_3, 
    \underbrace{b, \dots, b}_k ) \\
  &= \frakm_{k+l+1, \beta_1} ( \underbrace{-b, \dots, -b}_l, \bfe_3, 
    \underbrace{-b, \dots, -b}_k ),
\end{align*}
so that
\[
  \frakm_{1, \beta_2}^b (\bfe_3) = e^{-x}.
\]
Hence we have
\begin{equation*}
  \frakm_1^b (\bfe_3) 
    = \sum_{i=1}^2 \frakm_{1, \beta_i}^b (\bfe_3) T^{\beta_i \cap \omega}
    = e^x T^{\lambda +t} + e^{-x} T^{\lambda -t}.
\end{equation*}
Next we compute $\frakm_1^b(\bfe_1 \otimes \bfe_3) \in H^1(L_0)$.
Note that
\[
  \ev_{k+1}(\bfe_1 \otimes \bfe_3) 
  = (g^*\bfe_1 + f^* \bfe_1) \otimes (g^*\bfe_3 + f^* \bfe_3)
  = g^*\bfe_1 \otimes f^* \bfe_3 + \dots.
\]
Since only the term $g^*\bfe_1 \otimes f^* \bfe_3$ contribute to
$\frakm_{k+l+1, \beta_i}(b, \dots, b, \bfe_1 \otimes \bfe_3, b, \dots, b)$ 
by degree reason, we have
\[
  \frakm_{k+l+1, \beta_i} ( \underbrace{b, \dots, b}_k, \bfe_1 \otimes \bfe_3, 
    \underbrace{b, \dots, b}_l ) 
  = 
  \frakm_{k+l+1, \beta_i} ( \underbrace{b, \dots, b}_k, \bfe_1, 
    \underbrace{b, \dots, b}_l ) g^* \bfe_1.
\]
Hence we obtain
\begin{align*}
  \frakm_1^b (\bfe_1 \otimes \bfe_3) 
    &= \sum_{i=1}^2 \frakm_{1, \beta_i}^b (\bfe_1 \otimes \bfe_3) 
        T^{\beta_i \cap \omega} \\
    &= \sum_{i=1}^2 \frakm_{1, \beta_i}^b (\bfe_1) T^{\beta_i \cap \omega} \bfe_1 \\
    &= (e^x T^{\lambda +t} + e^{-x} T^{\lambda-t}) \bfe_1.
\end{align*}
\end{proof}

\begin{remark}
Iriyeh, Sakai, and Tasaki \cite{MR3127820} computed Floer cohomologies
$\HF(L, L' ;\bZ/2\bZ)$ of real forms in a compact Hermitian symmetric space,
i.e., fixed point sets $L = \Fix(\tau)$, $L' = \Fix(\tau')$ 
of anti-holomorphic and anti-symplectic involutions $\tau$, $\tau'$.
In particular, the Floer cohomology of the $\U(2)$-fiber $L_0 = \Fix(\tau_0)$ 
with coefficients in $\bZ/2\bZ$ is given by
\[
 \HF(L_0, L_0; \bZ/2 \bZ) \cong H^*(L_0; \bZ/2\bZ)
  \cong (\bZ/2\bZ)^4.
\]
On the other hand, \eqref{eq:Fl_diff_1} and \eqref{eq:Fl_diff_2}
implies that
\[
 \HF(L_0, L_0; \Lambda^{\bZ}_0) 
  \cong (\Lambda_0^{\bZ}/2 T^{\lambda} \Lambda_0^{\bZ})^2,
\]
where
\[
 \Lambda_0^{\bZ}
  = \left\{ \left.
      \sum_{i=1}^\infty a_i T^{\lambda_i} \, \right| \,
       a_i \in \bZ, \ 
       \lambda_i \ge 0, \ 
       \lim_{i \to \infty} \lambda_i = \infty
     \right\}
\]
is the Novikov ring over $\bZ$.
\end{remark}

\begin{remark}
Here we consider a Lagrangian $\U(n)$-fiber $L_t$ in $\Gr(n, 2n)$ for general $n$.
The one-parameter subgroup  $g_{\theta} = \exp (\theta \xi)$ of  $\U(2n)$
given by
\[
  \xi = \begin{pmatrix} 0 & -E_{11} \\ 
                   E_{11} & 0 \end{pmatrix}
  \in \fraku(2n)
\]
sends
\[
  x = \left( \begin{array}{ccc|ccc}
      t & & & \overline{x}^1_1 & \dots & \overline{x}^n_1 \\
        & \ddots & & \vdots & & \vdots \\
        & & t & \overline{x}^1_n & \dots & \overline{x}^n_n \\ \hline
      x^1_1 & \dots & x^1_n & -t & & \\
      \vdots & & \vdots & & \ddots & \\
      x^n_1 & \dots & x^n_n & & & -t
      \end{array} \right)
  \in L_t
\]
to $\Ad_{g_{\theta}} (x) \in \scO_{\bslambda}$ 
whose upper-left $n \times n$ block is given by
\[
  (\Ad_{g_{\theta}} (x) )^{(n)} =
  \begin{pmatrix}
    t (1- 2 \sin^2 \theta) - (x^1_1 + \overline{x}^1_1) \sin \theta \cos \theta & 
      - x^1_2 \sin \theta & \dots & -x^1_n \sin \theta \\
    - \overline{x}^1_n \sin \theta &  t  \\
    \vdots & & \ddots \\
    - \overline{x}^1_n \sin \theta & & & t
  \end{pmatrix}.
\]
If $\Ad_{g_{\theta}} (x)$ is still in $L_t$, i.e., 
$(g_{\theta} x g_{\theta}^*)^{(n)} = t I_n$, then we have
$x^1_2 = \dots = x^1_n = 0$ and
%$2t \sin \theta + (x^1_1 + \overline{x}^1_1) \cos \theta = 0$, i.e.,
%\[
$\re x^1_1 = - t \tan \theta$.
%\]
Since $|\re x^1_1| \le \sqrt{\lambda^2 - t^2}$, 
one has $g_{\theta} (L_t) \cap L_t = \emptyset$ if 
\[
  |\theta| > \arctan \sqrt{\frac{\lambda^2 - t^2}{t^2}}.
\] 
Note that the moment map $\mu : \scO_{\bslambda} \to \frak{u}(2n)$
of the $\U(2n)$-action is given by $\mu(x) = (\ii/2\pi) x$
in our setting.
Hence the Hamiltonian of $g_{\theta}$ is given by
\[
  H(x) = \frac{\ii}{2\pi} \langle x, \xi \rangle.
\]
Since 
$\max_{\scO_{\lambda}} H = \lambda / \pi$ and 
$\min_{\scO_{\lambda}} H = - \lambda / \pi$,
the norm of $g_{\theta}$ is given by
\[
  \int_0^{\theta} 
  \Bigl( \max_{\scO_{\lambda}} H - \min_{\scO_{\lambda}} H \Bigr)
  d \theta
  = \frac{2 \lambda}{\pi} \theta.
\]
Hence the displacement energy of $L_t$ is bounded from above by
\[
  h(t) =
  \frac{2 \lambda}{\pi} \arctan \sqrt{\frac{\lambda^2 - t^2}{t^2}}.
\]
Note that $h(t)$ is a concave function on $[-\lambda, \lambda]$ such that
$h(\pm \lambda) = 0$, $h(0) = \lambda$, and
$h(t) > \min \{ \lambda-t, \lambda+t \}$ for $t \ne 0, \pm \lambda$.
\end{remark}

\begin{theorem} \label{th:Gr(2,4)_2}
The Floer cohomology of the pair $(L_0, \pi \sqrt{-1}/2 \bfe_1)$,
$(L_0, -\pi \sqrt{-1}/2 \bfe_1)$ is given by
\[
  \HF( (L_0, \pm \pi \sqrt{-1}/2 \bfe_1), 
       (L_0, \mp \pi \sqrt{-1}/2 \bfe_1) ; \Lambda_0 )
  \cong (\Lambda_0 / T^{\lambda} \Lambda_0)^2.
\]
In particular, the Floer cohomology over the Novikov field is
trivial;
\[
  \HF( (L_0, \pm \pi \sqrt{-1}/2 \bfe_1), 
         (L_0, \mp \pi \sqrt{-1}/2 \bfe_1) ; \Lambda )
  = 0.
\]
\end{theorem}

\begin{proof}
For $b = \sqrt{-1}\pi/2 \bfe_1 \in H^1(L_0; \Lambda_0)$, 
we have from \eqref{eq:int_along_fiber} that 
\begin{align*}
  &\frakm_{k+l+1, \beta_i}(\underbrace{b, \dots, b}_k, \bfe_3, 
    \underbrace{-b, \dots, -b}_l) \\
  &\quad = \int_{(0, 2\pi)}
     \frac{1}{k!} \left( \frac{\sqrt{-1}}{4}  \theta \right)^k 
     \frac{1}{l!} 
     \left( \frac{\sqrt{-1}}{4}  \theta  - \frac{\pi \sqrt{-1}}{2} \right)^l
      (1 - \cos \theta) \frac{d \theta}{2\pi}.
\end{align*}
Hence the Floer differential is given by
\begin{align*}
  \delta_{b, -b} (\bfe_3) 
  &= \sum_{i=1, 2} \sum_{k, l\ge 0}
    \frakm_{k+l+1, \beta_i}(\underbrace{b, \dots, b}_k, \bfe_3, 
    \underbrace{-b, \dots, -b}_l) T^{\beta_i \cap \omega} \\
  &= 2 T^{\lambda} \int_0^{2\pi} \sum_{k, l \ge 0}
     \frac{1}{k!} \left( \frac{\sqrt{-1}}{4} \theta \right)^k 
     \frac{1}{l!} 
     \left( \sqrt{-1} \left( \frac{\theta}{4}  - \frac{\pi}{2} \right) \right)^l 
     (1 - \cos \theta) \frac{d \theta}{2\pi} \\
  &= 2 T^{\lambda} \int_0^{2\pi} e^{\sqrt{-1} (\theta / 2 - \pi /2)}  
    (1 - \cos \theta) \frac{d \theta}{2\pi} \\
  &= \frac{16}{3 \pi} T^{\lambda}.
\end{align*}
Similarly we have
\[
  \delta_{b, -b} (\bfe_1 \otimes \bfe_3) = \frac{32}{3 \pi} T^{\lambda} \bfe_1,
\]
and consequently,
\[
  \HF( (L_0,  \pi \sqrt{-1}/2 \bfe_1), 
       (L_0, - \pi \sqrt{-1}/2 \bfe_1) ; \Lambda_0 )
  \cong (\Lambda_0 / T^{\lambda} \Lambda_0)^2.
\]
The computation of 
$\HF( (L_0,  -\pi \sqrt{-1}/2 \bfe_1), 
       (L_0,  \pi \sqrt{-1}/2 \bfe_1) ; \Lambda_0 )$ is 
completely parallel.
\end{proof}

\bibliographystyle{amsalpha}
\bibliography{bibs}

\noindent
Yuichi Nohara

Faculty of Education,
Kagawa University,
Saiwai-cho 1-1,
Takamatsu,
Kagawa,
760-8522,
Japan.

{\em e-mail address}\ : \  nohara@ed.kagawa-u.ac.jp
\ \vspace{0mm} \\

\noindent
Kazushi Ueda

Department of Mathematics,
Graduate School of Science,
Osaka University,
Machikaneyama 1-1,
Toyonaka,
Osaka,
560-0043,
Japan.

{\em e-mail address}\ : \  kazushi@math.sci.osaka-u.ac.jp
\ \vspace{0mm} \\

\end{document}